\documentclass{amsart}
\usepackage{amsmath,amssymb,amsthm}
\usepackage[mathscr]{eucal}
\usepackage{amscd}
\usepackage{setspace}
\usepackage{graphicx}

\newtheorem*{maintheorem}{Theorem 5.3}

\newtheorem{theorem}{Theorem}[section]
\newtheorem{lemma}[theorem]{Lemma}
\newtheorem{cor}[theorem]{Corollary}
\newtheorem{prop}[theorem]{Proposition}

\theoremstyle{definition}

\numberwithin{equation}{section}

\begin{document}

\title[A Classification of Spanning Surfaces for Alternating Links]{A Classification of Spanning Surfaces for Alternating Links}
\date{\today}

\author[Colin Adams]{Colin Adams}
\address{Colin Adams,
Bronfman Science Center,
Williams College,
Williamstown, MA 01267}
\email{Colin.C.Adams@williams.edu}

\author[Thomas Kindred]{Thomas Kindred}
\address{Thomas Kindred,
3782 Olentangy Blvd.,
Columbus, OH 43214}
\email{thomas.kindred@gmail.com}

\begin {abstract}  A classification of spanning surfaces for alternating links is provided up to genus, orientability, and a new invariant that we call aggregate slope.  That is, given an alternating link, we determine all possible combinations of genus, orientability, and aggregate slope that a surface spanning that link can have.  To this end, we describe a straightforward algorithm, much like Seifert's Algorithm, through which to construct certain spanning surfaces called layered surfaces.  A particularly important subset of these will be what we call basic layered surfaces. We can alter these surface by performing the entirely local operations of adding handles and/or crosscaps, each of which increases genus. 

The main result then shows that if we are given an alternating projection $P(L)$ and a surface $S$ spanning $L$, we can construct a surface $T$ spanning $L$ with the same genus, orientability, and aggregate slope as $S$ that is a basic layered surface with respect to $P$, except perhaps at a collection of added crosscaps and/or handles.  Furthermore, $S$ must be connected if $L$ is non-splittable.

This result has several useful corollaries.  In particular, it  allows for the determination of nonorientable genus for alternating links. It also can be used to show that mutancy of alternating links preserves nonorientable genus. And it allows one to prove that there are knots that have a pair of minimal nonorientable genus spanning surfaces, one boundary-incompressible and one boundary-compressible.
\end{abstract}
\maketitle
\footnotetext[1]{2000 Mathematics Subject Classification 57M25, 57M50 Key 
words: spanning surface, alternating knot, crosscap number}
\section{Introduction}

A Seifert surface for an oriented knot or link  $L$ is a compact connected oriented surface in $S^3$ with oriented boundary equal to $L$. Seifert's Algorithm, wherein one takes an oriented projection of the link,  cuts each crossing open in the manner that preserves orientation, fills in each resulting circle with a disk, and connects the disks with a half-twisted band at each crossing,  ensures that such a surface exists. The minimal genus of a Seifert surface is defined to be the genus of the knot or link. Given a random link, it is typically difficult to determine its genus. However for one category of link, we have a simple means for finding the genus.

\begin {theorem}\label{Gabai}Given an oriented alternating link, Seifert's Algorithm applied to a reduced alternating projection yields a minimal genus Seifert surface.
\end{theorem}

This theorem has been proven three times, first independently  in 1958 by Kunio Murasugi \cite{MK} and in 1959 by Richard Crowell \cite{CR} using the reduced Alexander Polynomial, and then in 1987 by David  
Gabai \cite{G} using elementary geometric techniques.   Our main theorem will include this theorem as a special case.

However, in this paper, we will also be concerned with nonorientable spanning surfaces. The \textbf{crosscap number} of a nonorientable surface $S$ of $n$ boundary components is defined to be  
$2 -\chi(S) - n$. Note that this is the number of projective planes of which the corresponding closed surface is the connected sum. In \cite{CB}, the \textbf{crosscap number of a link $L$} is defined to be the minimum crosscap number of any nonorientable surface spanning the link $L$. 

For the sake of comparison between complexities of orientable and nonorientable surfaces, define the \textbf{nonorientable genus} of $L$ to be 1/2 of the crosscap number of $L$. Note that this can be either an integer or a half integer, and does not coincide with traditional definitions. We will sometimes refer to the genus of a knot or link as the orientable genus if confusion may otherwise arise. 

The crosscap number is known only for a relatively small collection of knots and links. In   \cite{T}, it was determined for torus knots (see also \cite{MS}).  In \cite{HT}, a simple method for its determination was provided for two-bridge knots and links.  In \cite{IM}, it was determined for many pretzel knots.

In this paper, we present a classification of spanning surfaces for alternating links up to genus, orientability, and a new invariant that we call aggregate slope, and that for knots coincides with the slope.  That is, given an alternating link, we determine all possible combinations of genus, orientability, and aggregate slope that a surface spanning that link can have.  To this end, we describe a straightforward algorithm, much like Seifert's Algorithm, through which to construct spanning surfaces.  We call the surfaces generated by this algorithm \textbf{layered surfaces}.  A particularly important subset of these will be what we call \textbf{basic layered surfaces}.

Once we have a layered surface, we can also change the surface by performing the entirely local operations of adding handles and/or crosscaps, each of which increases genus. 

Our main result, which appears in Section 5,  is as follows:

\begin{maintheorem}
Given an alternating projection $P(L)$ and a surface $S$ spanning $L$, we can construct a surface $T$ spanning $L$ with the same genus, orientability, and aggregate slope as $S$ such that $T$ is a basic layered surface with respect to $P$, except perhaps at a collection of added crosscaps and/or handles.  When $S$ is orientable, $T$ can be chosen to be orientable with respect to the orientation that $L$ inherits from $S$.
\end{maintheorem}

This result has several useful corollaries that appear in Section \ref{S:Implications}, including the already mentioned classification of spanning surfaces for alternating links.  Additionally, it allows us to easily find and construct a minimal nonorientable genus surface spanning a given alternating link.  Note, for example in \cite{Z}, that a substantial amount of work is necessary to show the crosscap number of the link $6_3^2$ is 3. Theorem \ref{main} and the algorithm given below make this an immediate conclusion. At the end of the paper, we include the nonorientable genus of all the alternating knots of nine or fewer crossings and alternating 2-component links of eight or fewer crossings. It would be straightforward to write a computer program to determine the nonorientable genus of all knots in the census of alternating knots through 22 crossings given in \cite{HTW},\cite{RFSI} and \cite{RFSII}.

In \cite{HT}, Hatcher and Thurston proved that a 2-bridge knot cannot have two minimal nonorientable genus  spanning surfaces, one boundary-incompressible and one boundary-compressible. They then asked whether or not this is true in general. We utilize the results of this paper to answer that question in the negative with a specific example.

In a recent paper (cf. \cite{CT}),Curtis and Taylor give a further corollary of Theorem \ref{main}, showing that the two checkerboard surfaces of a reduced alternating projection of an alternating knot must yield the maximum and minimum integral slopes for all essential boundary surfaces of the knot. This implies that these maximum and minimum integral slopes must be twice the number of positive crossings and the negative of twice the number of negative crossings in the projection respectively. Since all slopes are integral for Monesinos knots, this yields the maximum and minimum slopes for all surfaces in that case.

The layered surfaces that appear here have independently been considered previously in \cite{PPS} (see footnote on p.10 of the ArXiv version) and in \cite{O}, where they are called state surfaces and where a criterion is provided that proves certain such surfaces are essential.

\section{Background}

Standard definitions of reduced projections, flypes, positive and negative crossings, oriented links, incompressible and boundary-incompressible surfaces apply.  A surface that is  
incompressible \emph{and} boundary-incompressible is said to be \textbf{essential}. 

If a projection $P$ contains $p$ positive crossings and $n$ negative crossings,  
then the \textbf{writhe} $w(P)$ is ${p-n}$.  Writhe varies across projections of a given link and across orientations of a given projection, but is invariant across all reduced, alternating projections of an oriented link.

 We define the \textbf{aggregate linking number} $lk(L)$ to be the sum of the linking numbers of all pairs of link components.  We define the aggregate linking number of a knot to be zero.  Equivalently, the aggregate linking number can be calculated in the same way as writhe, but only counting crossings between \emph{distinct} link components, and halving at the end.   Linking number and aggregate linking number are invariant across all projections of a given oriented link, but may well vary across various orientations of a given link.
 
 A projection divides the projection surface $S^2$ into regions, each of whose boundary consists alternately of strands and crossings from $L$.  A region with $n$ crossings on its boundary will be called a \textbf{projection n-gon}. 

 A \textbf{spanning surface} for a link $L$ is a surface $S \subset S^3$ with boundary equal to $L$ such that $S \cap  
\partial N(L)$ consists of one $(p, 1)$-curve on each component of $ 
\partial N(L)$, where $p$ refers to the number of meridians.  We will not distinguish between the spanning surface $S$ as a subset of $S^3$ and the spanning surface as a subset of $S^3 \backslash N(L)$, calling both $S$.  Two spanning surfaces for $L$ are considered equivalent if one is ambient isotopic to the other in $S^3 \backslash N(L)$. We will assume that $S$ contains no closed components, so that every point in $S$ is connected to $L$ by a path in $S$.  

We need a special form of boundary-incompressibility for spanning surfaces.

A spanning surface $S$ in $S^3 \backslash N(L)$  is \textbf{meridianally boundary-compressible} if there exists a disk $D$ embedded in $S^3 \backslash N(L)$ such that $\partial D = \alpha \cup \beta$ where $D \cap S = \beta$ and $D \cap \partial N(L) = \alpha$ are both arcs, $\beta$ does not cut a disk off $S$, $\partial D \cap \partial S$ cuts $\partial S$ into two arcs $\phi_1$ and $\phi_2$,  and $\alpha \cup \phi_i$ is a meridian of the knot for at least one of $i = 1,2$. A spanning surface is said to be \textbf{meridianally boundary-incompressible} if no such disk exists. The spanning surface is said to be \textbf{meridianally essential} if it is incompressible and meridianally boundary-incompressible.

The \textbf{boundary slope} or just \textbf{slope} $l(S,\hat{L})$ of a  
spanning surface $S$ with respect to a particular oriented component  
$\hat{L}$ of the link it spans is the linking number of $\hat{L}$ with $S \cap \partial N(\hat{L})$, its parallel curve on $S$ (where $S\cap \partial N(\hat{L})$ inherits the orientation of $\hat{L}$, as we will henceforth assume it to do). The \textbf{aggregate slope} $l(S,L)$ of a surface $S$ spanning oriented link $L$ is the sum of the link components' individual slopes.  Note that slope and aggregate slope are independent of link orientation.

We define the \textbf{twist} of a spanning surface $S$  
relative to projection $P$ to be $\tau(S,P)=l(S,L) - (w(P)-2 lk(L))$, where $l(S,L)$ is the aggregate slope, $w(P)$ is the writhe, and $lk(L)$ is the aggregate linking number of $L$.

\begin{prop}

The twist $\tau(S,P)$ can be calculated by taking the projection of $L \cup(S \cap \partial N(L))$ and considering only those crossings that occur between edges of the annuli in $S \cap N(L)$.  If we have $p$ positive crossings of this type and $n$ negative crossings, then $\tau(S,P)=\frac{p-n}{2}$.

\end{prop}

\begin{proof}

First, notice that $S \cap N(L)$ is a collection of annuli.  Holding $L$ fixed, we can isotope each annulus in $S \cap N(L)$ so that in the projection of $L \cup(S \cap \partial N(L))$, any crossing between some component of $S \cap \partial N(L)$ and its respective boundary through an annulus in $S \cap N(L)$ \emph{does not} occur in a neighborhood of a crossing of $L$ in the projection.  The projection of $L \cup(S \cap \partial N(L))$ will then have the following seven types of crossings, where $L_i$ and $L_j$ denote any distinct pair of link components and $\hat{L_i}$ and $\hat{L_j}$ their respective components in $\hat{L}=S\cap \partial N(L)$:

(I) $L_i$ with itself

(II) $L_i$ with $L_j$

(III) $L_i$ with $\hat{L_i}$, such that the pre-image of the crossing is an arc connecting the two boundary components of the annulus $S \cap N(L_i)$

(IV) $L_i$ with $\hat{L_i}$ in a neighborhood of a crossing of $L$, rather than  
through $S \cap N(L_i)$

(V) $\hat{L_i}$ with itself, necessarily in the neighborhood of  a crossing of $L$.

(VI) $L_i$ with $\hat{L_j}$, necessarily in the neighborhood of  a crossing of $L$.

(VII) $\hat{L_i}$ with $\hat{L_j}$, necessarily in the neighborhood of  a crossing of $L$.

Notice, as in Figure \ref{Fi:crossing}, that the neighborhood of each Type I crossing will also contain exactly two Type IV crossings and one Type V crossing, each of the same type (positive or negative).  The same will also be true of Types II, VI, and VII, respectively.  

\begin{figure}
\begin{center}
	{\includegraphics[height=1.3in]{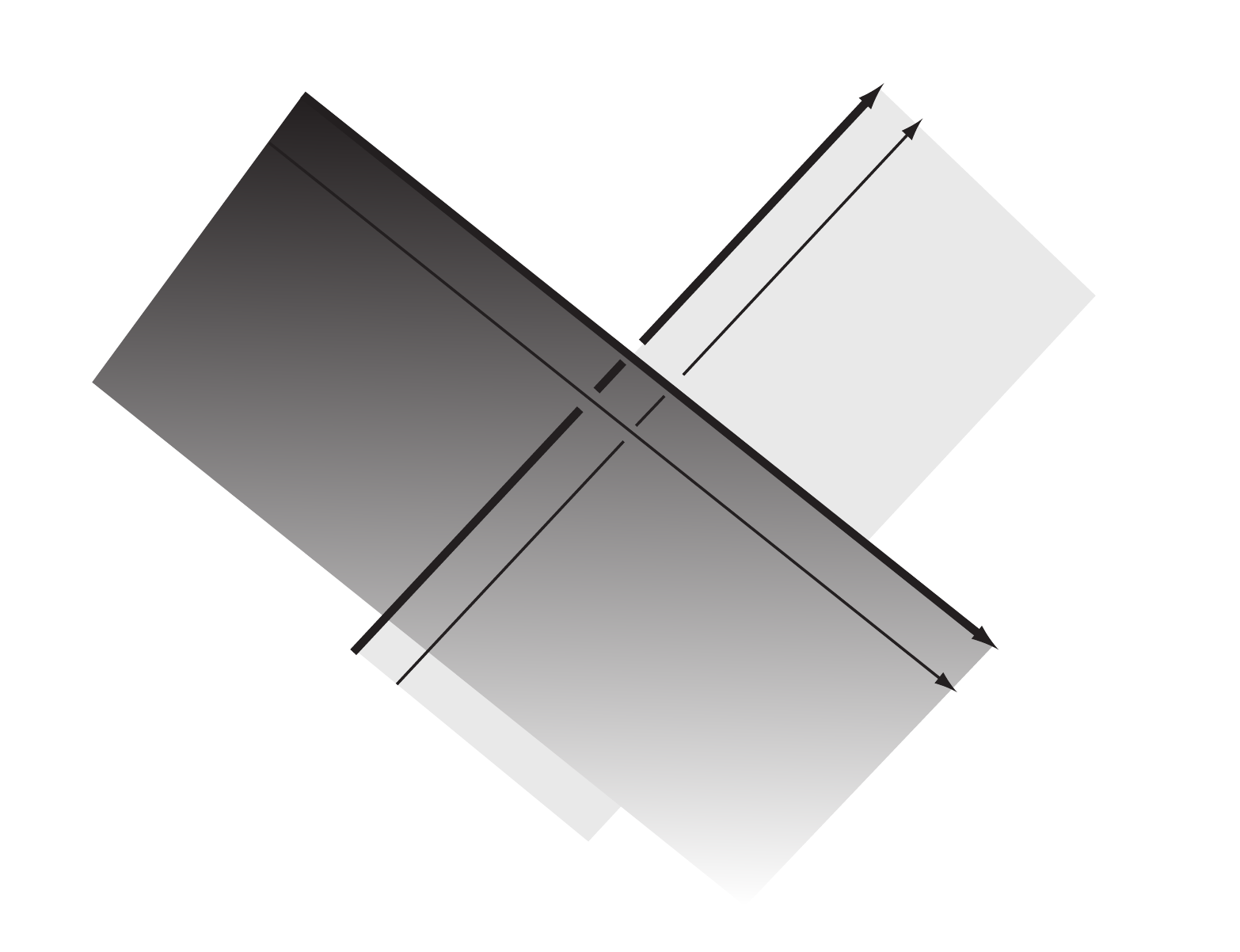}}
	\caption{The projected neighborhood of each Type I crossing, necessarily containing two Type IV crossings and one Type V crossing.}\label{Fi:crossing}
\end{center}
\end{figure} 

We then obtain the following expression for aggregate slope, where $I$ denotes the number of positive Type I crossings minus the number of negative ones, and similarly for $II$, $III$, and $IV$.

$$l(S,L)=\frac{III+IV}{2}=\frac{III}{2}+\frac{2I}{2}=\frac{III}{2}+((I+II)-2(\frac{II}{2}))=\frac{III}{2}+(w(P)-2lk(L))$$

Rearranging gives:

$$\frac{III}{2} = l(S,L)-(w(P)-2lk(L))=\tau(S,P)$$

as desired.
\end{proof}

We also use this classification of the seven types of crossings in $L \cup \hat{L}$ to prove the following.

\begin{prop}

Any spanning surface has an even aggregate slope.

\end{prop}

\begin{proof}

We proceed by induction on Euler Characteristic, and assume $S$ to be connected, spanning $L$ (we can easily generalize to disconnected surfaces).  If $\chi(S)=1$, $S$ is a disk, so aggregate slope is zero.  Assume the proposition is true for $\chi(S) > -n$.  Let $\chi(S)=-n$.  Find an arc $\alpha$ on $S$ such that $S'=S \backslash N(\alpha)$ is connected, spanning a link $L'$.  Since $\chi(S')=1-n$, it must have an even aggregate slope, by our inductive hypothesis.  Now, isotope $S$ and $L$ to shorten $\alpha$, changing the projection to give the picture in Figure \ref{Fi:band}.  Call the resulting projection of the link $P$. 

\begin{figure}
\begin{center}
	\scalebox{.5}{\includegraphics[height=1.3in]{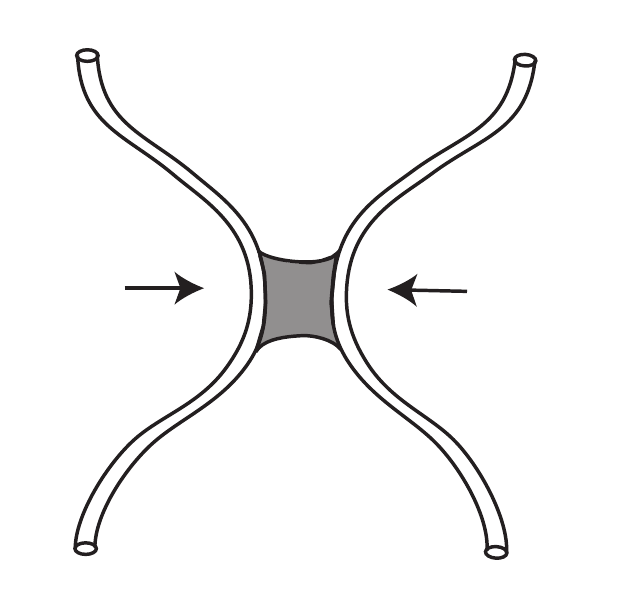}}
	\caption{Shrinking $\alpha$.}\label{Fi:band}
\end{center}
\end{figure} 

Cutting along $N(\alpha)$ will not create any new crossings, so any change in aggregate slope will result from positive crossings becoming negative crossings or vice-versa, as a result of changes in orientation.  Since $\hat{L}$ will inherit the orientation of $L$, all positive type III crossings will remain positive, and all negative type III crossings will remain negative.  The only other crossings that contribute to aggregate slope are type IV crossings, but each type IV crossing will appear as in Figure \ref{Fi:crossing}, from which it follows that any change in the orientation of type IV crossings will increase or decrease the aggregate slope by exactly two.  Therefore, if $l(S',L')$ is even, then $l(S,L)$ will be as well, proving our statement.
\end{proof}

Although writhe and linking number may vary with the orientation of a link, the slope, aggregate slope, and twist will not.

\begin{prop}

Slope, aggregate slope, and linking number are invariant across all projections of isotopic embeddings of a given surface $S$ spanning a given oriented link $L$, but writhe and twist are not.  Writhe and twist are invariant across all projections of isotopic embeddings of $S$ in which $L$ has a given orientation and is in reduced, alternating form.

\end{prop}

\begin{proof}

Isotopy of $S$ does not change the identity of the link $L \cup(S \cap \partial N(L))$, and since linking number is invariant across all projections, slope, aggregate slope, and linking number are invariant across all projected embeddings of $S$.  Adding a half-twist to $P(L)$ alters the writhe, and the twist as well, since $\tau(S,P)=l(S,L) - (w(P)-2lk(L))$.  Because writhe is invariant under flyping, it is invariant across all reduced, alternating projections of an oriented link $L$.  Therefore, twist is as well.
\end{proof}

Note that compression increases $\chi(S)$ by two, and boundary compression increases $\chi(S)$ by one. 

\begin{prop}A meridianal boundary-compression of a spanning surface yields another spanning surface.
\end{prop}

\begin{proof}
Suppose that  surface $S$ spans link $L$ and that $D$ is a meridianal boundary-compression disk.  Then $\textit{\r{D}} \cap S = \varnothing$ and $\partial D = \alpha \cup \beta$ where $\alpha \subset \partial N(\hat{L})$ for some link component $\hat{L}$ and $\beta \subset S \backslash N(\hat{L})$.   Note that $\alpha$ intersects $S$ only at its endpoints, and that $S \cap N(\hat{L})$ consists of two arcs sharing the endpoints of $\alpha$.  We see then that we can isotope $D$ through $S^3 \backslash S$, with its boundary moving across $S \cup \partial N(\hat{L})$, such that we obtain the general picture of a meridianal boundary compression shown in Figure \ref{compression} (in fact, there are two possible pictures of a mertidianal boundary compression, the second being the reflection of the one shown in the figure).  

Performing the boundary compression amounts to cutting $S$ along $N(D) \cap S$, and then gluing in two parallel copies of $D$.  The picture shows that this yields a spanning surface, whose slope along $\hat{L}$ differs from that of $S$ by $\pm 2$.  

\end{proof}

\begin{figure}
\begin{center}
	\scalebox{.3}{\includegraphics{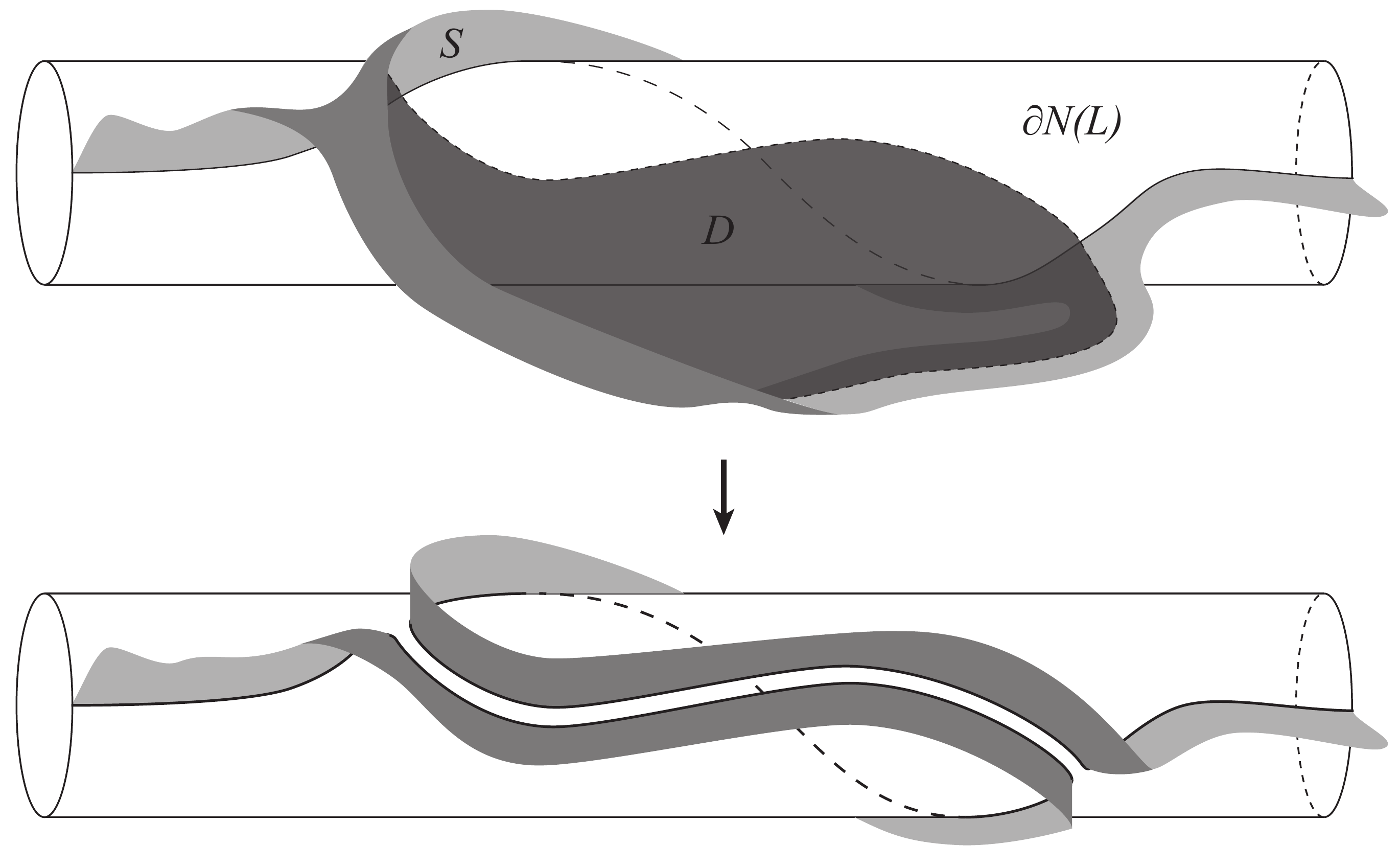}}
	\caption{Meridianal boundary-compression of a spanning surface yields another spanning surface.}\label{compression}
\end{center}
\end{figure}

Note that the meridianal boundary compression yields another spanning surface with slope that has increased or decreased by exactly two, while compression leaves the slope unchanged.  Since neither compression nor meridianal boundary compression affects the writhe or linking number of an oriented projection, meridianal boundary compression increases or decreases twist by exactly two and compression leaves twist unchanged.

Seifert's Algorithm for producing a Seifert surface for an arbitrary link  need not generate a minimal genus Seifert surface, as the example in Figure \ref{Fi:compress} demonstrates. 

\begin{figure}
\begin{center}
	\scalebox{.5}{\includegraphics{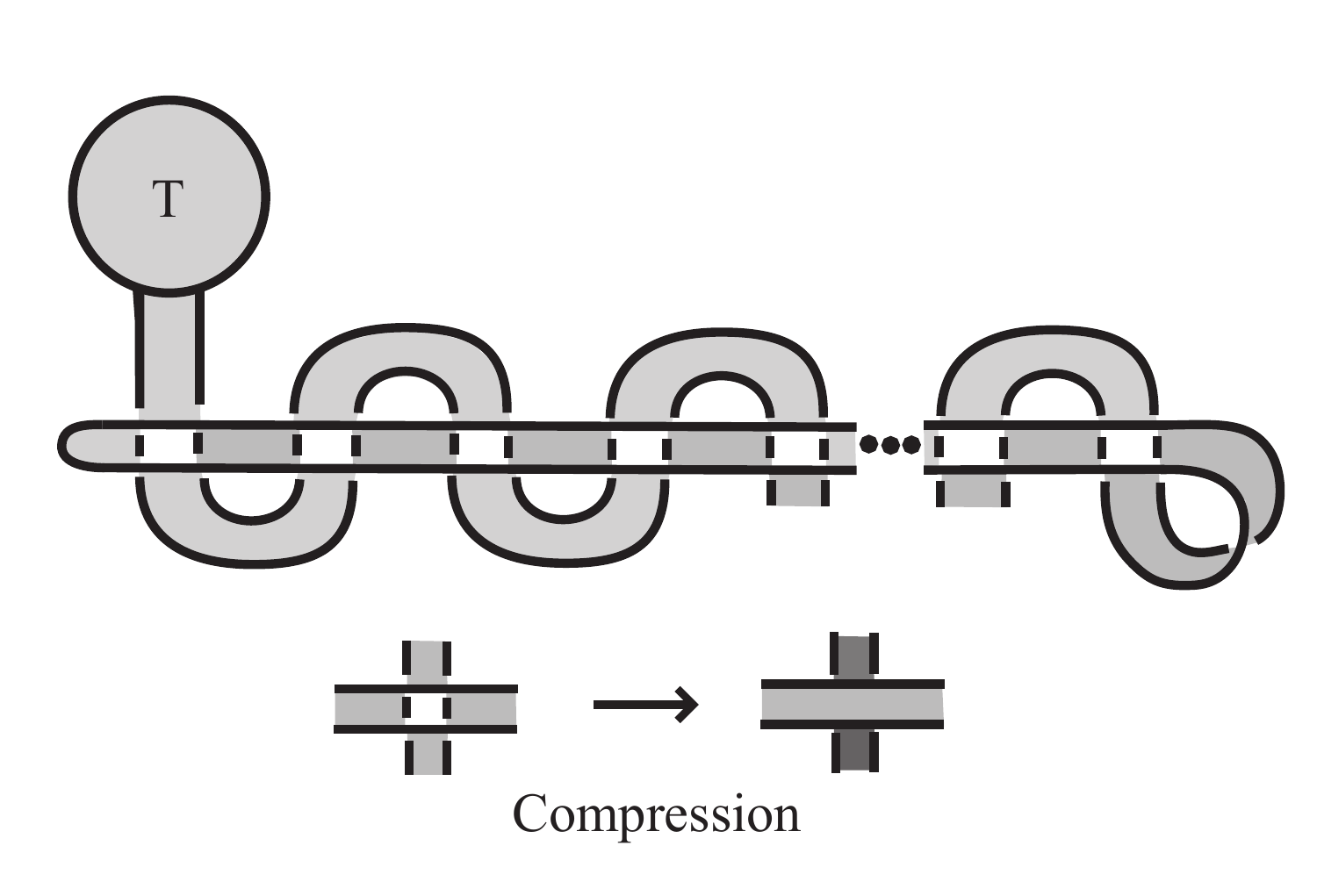}}
	\caption{Seifert's Algorithm can generate a surface that can be compressed an arbitrarily large number of times.}\label{Fi:compress}
\end{center}
\end{figure} 

However, Theorem \ref{Gabai} shows that it does when the link is alternating. One would like to obtain a similar result in the case of nonorientable spanning surfaces, so it is worth considering Gabai's method of proof.

 First, he shows that given orientable surfaces $S$ and $T$, each spanning an oriented link $L$, with $g(T)<g(S)$, there exists an orientable surface $T'$ with smaller genus than $S$ such that \textit{\r{T'}} $\cap$ \textit{\r{S}}$=\varnothing$.  Second, he takes a surface $S$ obtained by applying Seifert's Algorithm to an $n$-crossing oriented link and assumes there exists an orientable spanning surface of smaller genus.  He then uses the earlier result to construct a surface $T$ of smaller genus than $S$, where the two are disjoint in their interiors.  He then uses this structure to find arcs on $S$ and $T$ that are parallel to the same crossing, and he cuts along each.  This reduces by one the crossing number of their shared boundary and increases by one each surface's Euler characteristic.  He now has surfaces $S'$ and $T'$ spanning an $(n-1)$-crossing alternating link, where $S'$ comes from Seifert's Algorithm and $g(T')>g(S')$.  From here, he applies an inductive argument on crossing number and reaches a contradiction. 
\bigskip

The proof is elegantly elementary but fails to generalize to nonorientable surfaces because of the following.

\bigskip

\begin{prop} 

No two spanning surfaces of the same link, at least one of which is nonorientable, can be disjoint in their interiors.

\end{prop}

\bigskip

\begin{proof}

Given two surfaces disjoint in their interiors, spanning the  
same link, they must have the same boundary slopes. We can then connect their two boundaries with an annulus in $\partial N(L)$ to obtain a single, closed surface without boundary, embedded in 3-space.  If either spanning surface is nonorientable, then the closed surface will be as well.  Since no closed, nonorientable surface can be embedded in 3-space, we have a  
contradiction.
\end{proof}

Before discussing nonorientability further, we present one more useful fact.  Suppose a surface $S$ spans an oriented link $L$.  For any arc $\alpha \subset S$, with both endpoints on $L$, $N(\alpha) \cap S$ will take one of the forms pictured in Figure \ref{Fi:orientbands}.  We will call the former a \textbf{de-orienting band} and the latter an \textbf{orienting band}.

\begin{figure}
\begin{center}
	\scalebox{.5}{\includegraphics{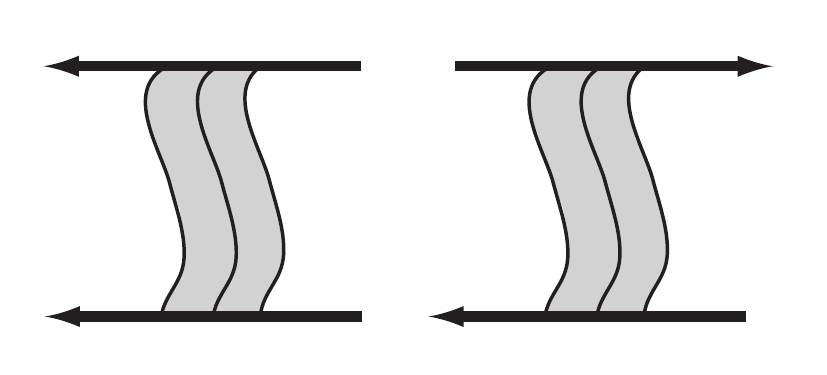}}
	\caption{A de-orienting band and an orienting band.}\label{Fi:orientbands}
\end{center}
\end{figure}

\begin{prop}

A surface $S$ spanning $L$ is orientable if and only if it is possible to orient $L$ so that $S$ contains no de-orienting bands.  Equivalently, $S$ is nonorientable  if and only if $S$ contains a de-orienting band under any orientation of $L$.

\end{prop}

\begin{proof}

If a surface is orientable, its boundary inherits an orientation from an orientation of the entire surface. The existence of a de-orienting band would contradict the orientation inherited by the boundary. Conversely, if the surface is non-orientable it contains a Mobius band. Let $\beta \subset S$ be the core curve of this Mobius band and let $\gamma_1$, $\gamma_2 \subset S$ be arcs with one endpoint on $\beta$ and the other on $L$.  The endpoints of $\gamma_1$ and $\gamma_2$ on $\beta$ divide it in two.  Let $\beta_1$ and $\beta_2$ be these two arcs of $\beta$, as in Figure \ref{Fi:Mobius}.  

\begin{figure}
\begin{center}
	\scalebox{.5}{\includegraphics{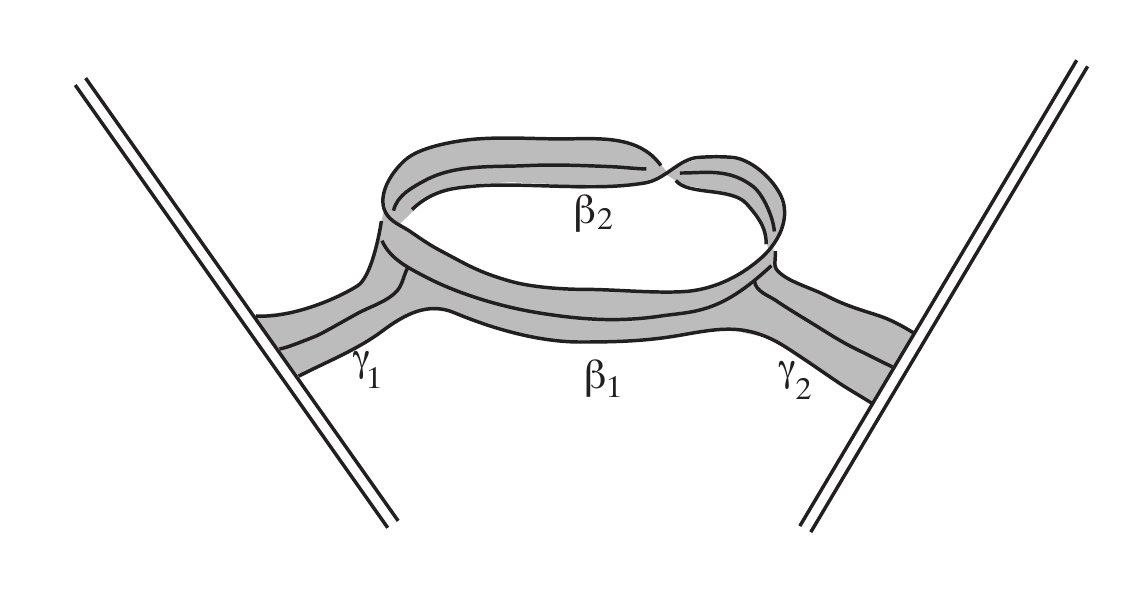}}
	\caption{$N(\gamma_1 \cup \beta \cup \gamma_2) \cap S$}\label{Fi:Mobius}
\end{center}
\end{figure} 

Then under any orientation of $L$ either $N(\gamma_1 \cup \beta_1 \cup \gamma_2)$ or $N(\gamma_1 \cup \beta_2 \cup \gamma_2)$ will be a de-orienting band. 
\end{proof}

This will be an important fact in the proof of Theorem \ref{main}.  As for the matter at hand, we say that a spanning surface $S$ is orientable \emph{relative} to its oriented boundary link $L$ if $L$ is oriented in the way guaranteed by the above result.  Note that any orientable surface $S$ can be oriented (assigned a normal direction) in two ways, and each of these orientations defines an orientation on its link boundary, $L$.  These orientations of $L$ are precisely the ones relative to which $S$ is oriented.  The example in Figure \ref{Fi:annulus} may be useful in seeing how this distinction is important for Theorem \ref{main}, since this surface is orientable (it contains no Mobius band) but \emph{not} relative to its oriented boundary (it contains de-orienting bands).  

\begin{figure}
\begin{center}
	\scalebox{.5}{\includegraphics{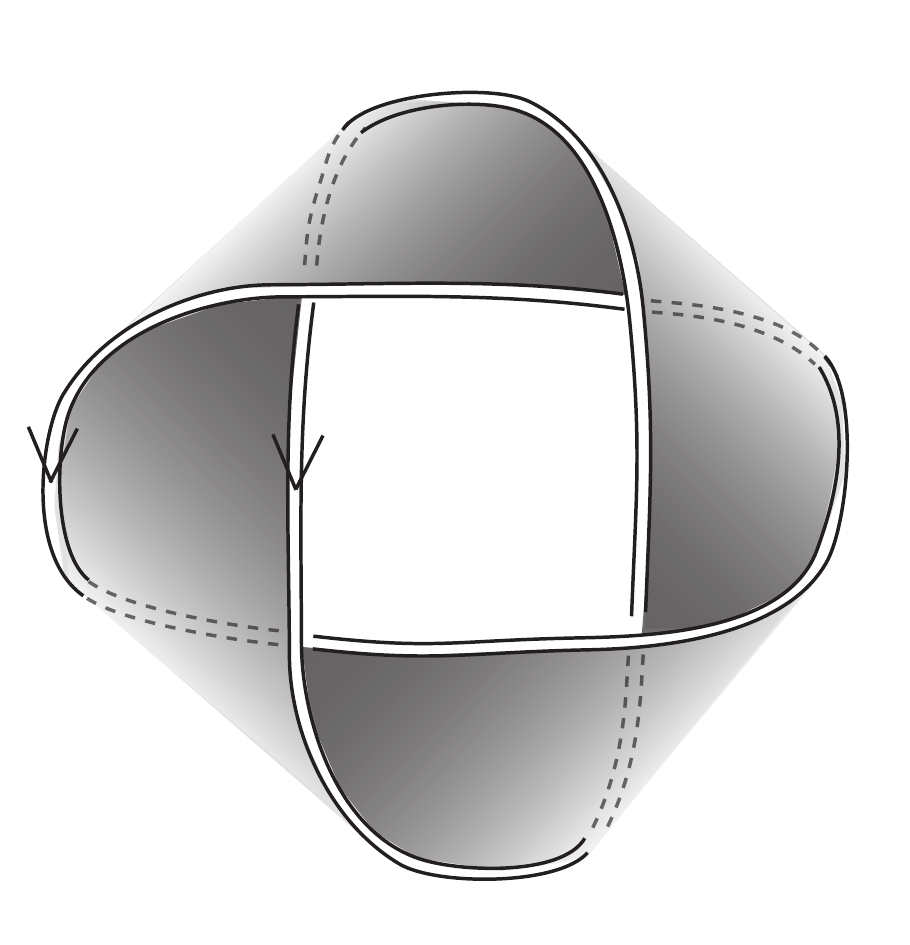}}
	\caption{This annulus spanning the $4_1^2$ link is orientable because it contains no Mobius band, but not with respect to its oriented boundary, since it contains a de-orienting band.}\label{Fi:annulus}
\end{center}
\end{figure} 

\section {Layered Surfaces} \label{Sec:layered}

\subsection{Construction}
In \cite{A5},  checkerboard  
surfaces generated from reduced, alternating projections were shown to be  
essential, non-accidental, and non-fibered and therefore quasi-Fuchsian surfaces in hyperbolic link complements. We seek to describe a larger class of essential, non-accidental, nonorientable  surfaces spanning these links.  Since fibers must be orientable, any such surfaces must be quasi-Fuchsian. To this end, we devised the following natural extension of Seifert's Algorithm, which we call the Layered Surface Algorithm, as depicted in Figure \ref{Fi:layered}.

\bigskip

(1) Beginning with an alternating projection $P(L)$, split open each  
crossing in one of the two directions. 

(2) This will result in a collection of non-overlapping, possibly  
nested circles.  Choose a region on the projection sphere to contain ${\infty}$. Put nested circles at different heights relative to  
$S^2$ and fill each circle with a disk to the side that does not contain ${\infty}$ when projected to $S^2$.

(3) Connect the disks with half-twist bands at the crossings, yielding a  
surface with boundary equal to the link such that the link projects to $P(L)$, as in Seifert's Algorithm, if it is possible to do so without a crossing band intersecting an existing disk. If this is not possible, change the heights of the disks so that the disks can be connected with crossing bands.

\begin{figure}
\begin{center}
	\scalebox{.5}{\includegraphics{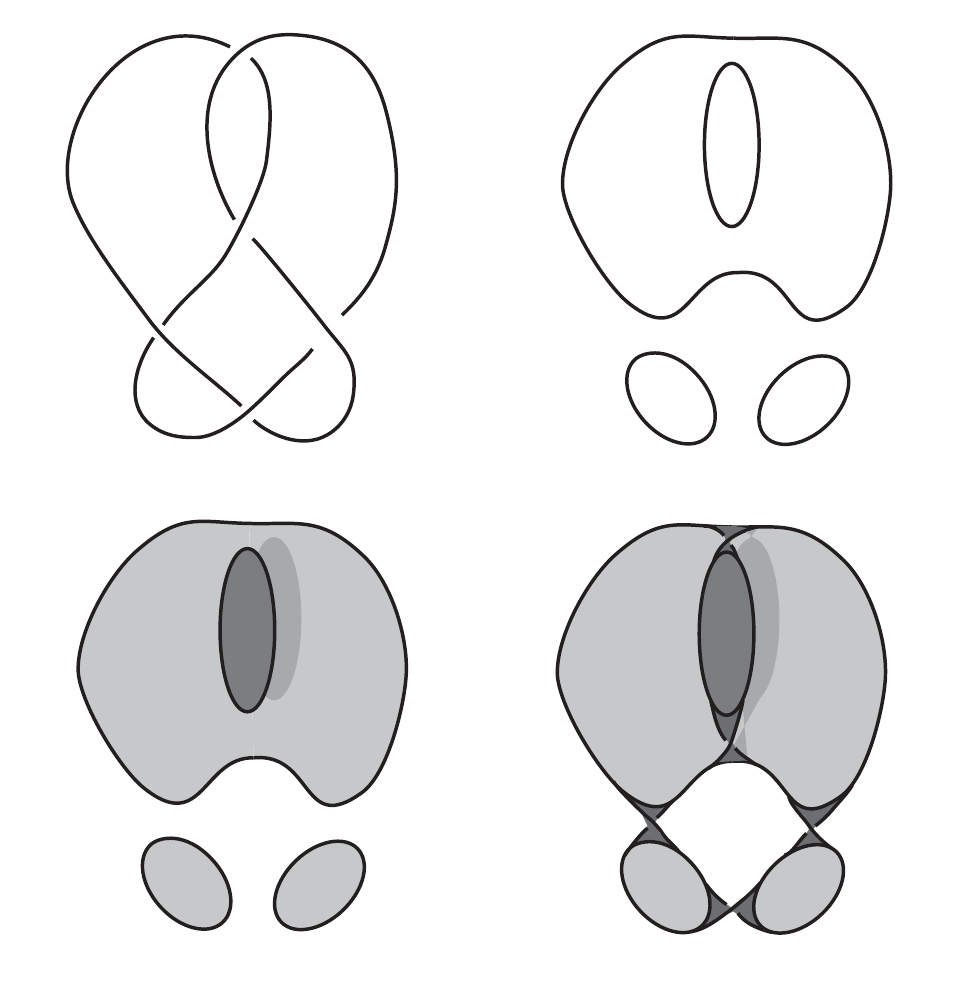}}
	\caption{Constructing a layered surface.}\label{Fi:layered}
\end{center}
\end{figure} 

Not all choices of relative disk heights can generate a layered surface, as the crossing band between two disks might intersect a third disk.  It is always possible, though, to choose heights for the disks to generate a layered surface, most simply by stacking all inner disks (disks with boundary circles inside boundary circles of other disks) above outer ones.  Note that while different choices of relative heights for nested circles may well generate distinct layered surfaces, the choices of disk heights will not affect Euler characteristic, orientability, or aggregate slope.  

Layered surfaces that have a crossing band that connects a circle to itself are often boundary-compressible (in fact, meridianally boundary-compressible). To avoid this issue, we define a \textbf{basic layered surface} to be one that does not have a crossing band connecting a circle to itself. For our purposes, the only layered surfaces with which we will be concerned are basic layered surfaces. In a subsequent paper, we plan to discuss the questions of essentiality and non-accidentality of these surfaces.

\subsection{Properties}

The greatest initial evidence that  
\emph{something like} Theorem \ref{main} might be true was that for  
small crossing number (all of the knots are 2-bridges), layered surfaces seemed to match up perfectly (at least according to genus, orientability and boundary-slope)  
with Hatcher and Thurston's surfaces from \cite{HA}, which they proved to be a complete classification of essential surfaces spanning these knots.  

In 1991, W. Menasco and M. Thistlethwaite used Menasco's geometric structure to prove Tait's Flyping Conjecture (cf. \cite{M3}, \cite{M6}), which states that any two reduced, alternating projections of the same link are related through a finite sequence of flypes.  Let us consider what happens to layered surfaces under  
flyping.

\begin{prop}

Given a layered surface $S$ generated from an alternating projection $P$ and a flype that takes $P$ to projection $P'$, there exists a layered surface $S'$ obtained from $P'$ with the same genus, orientability, and aggregate slope as $S$, where $S'$ is a basic layered surface if and only if $S$ is.  

\end{prop}

\begin{proof}

See Figure \ref{Fi:flypes}. We have eight possible situations, depending on the crossing splits. After splitting, there are two possibilities for the single crossing; it can be split vertically or horizontally. After splitting in T, there are two possibilities. It could be the case that C shares a circle  with D and E shares a circle with F.  Or it could be that  C shares a circle with E and D shares a circle with F. Finally, there are two  possibilities for the splittings outside this part of the projection. Either A is connected to E and B to F, or A is connected to B and E to F.
Three of the eight possibilities yield figures equivalent to three others, yielding five figures to consider. For each, we perform a flype about the tangle $T$.  Since $S$ is layered, $S$ is layered both inside and outside of $T$. In each case, we can obtain a surface after the flype with the same genus, orientability and aggregate slope as the original surface. Note that the original and resulting surface are either both basic or neither basic. Also note that the last two cases are never basic layered surfaces and always nonorientable. 

\begin{figure}
\begin{center}
	\scalebox{.8}{\includegraphics{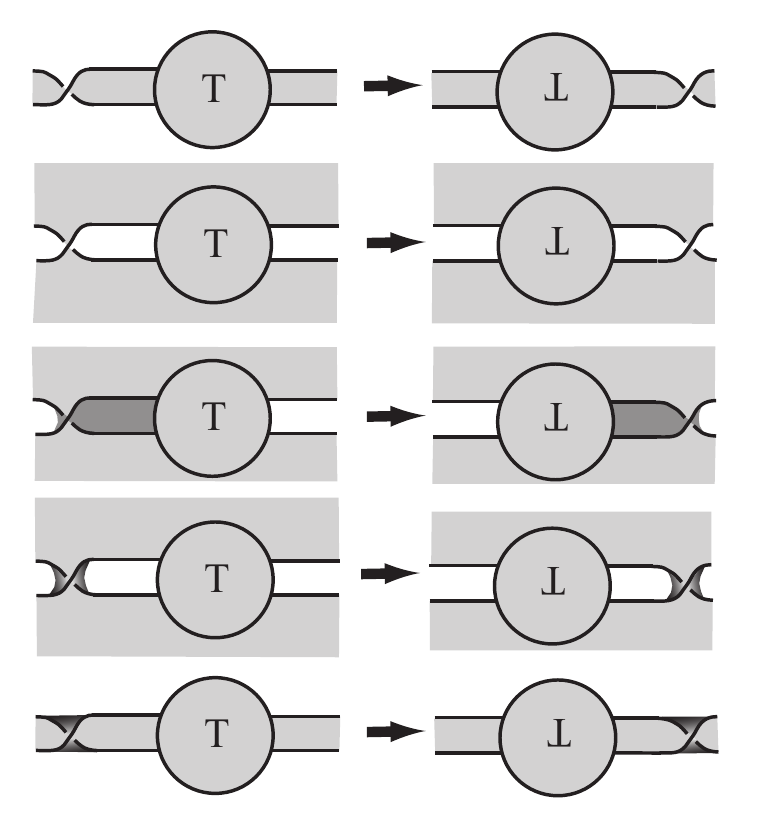}}
	\caption{The five possible appearances of a layered surface relative to a flype, together with the layered surface for the resulting projection that has the same genus, orientability, and aggregate slope as the original.}\label{Fi:flypes}
\end{center}
\end{figure} 
\end{proof}

So any two alternating projections of the same link  
generate equivalent collections of layered surfaces, up to genus, orientability and  
aggregate slope.

For our next fact, notice that in the construction of a layered surface we have two options in the splitting of each crossing, which we will call A and B splits, as in Figure \ref{Fi:splits}. As is standard(cf.\cite{A}, Chapter 6), the regions adjacent to a crossing can be labelled A and B depending on whether we twist the top strand counterclockwise or clockwise to cover those regions. The A or B-split corresponds to creating a channel between the two regions with that label at the crossing.  

\begin{figure}
\begin{center}
	\scalebox{.8}{\includegraphics{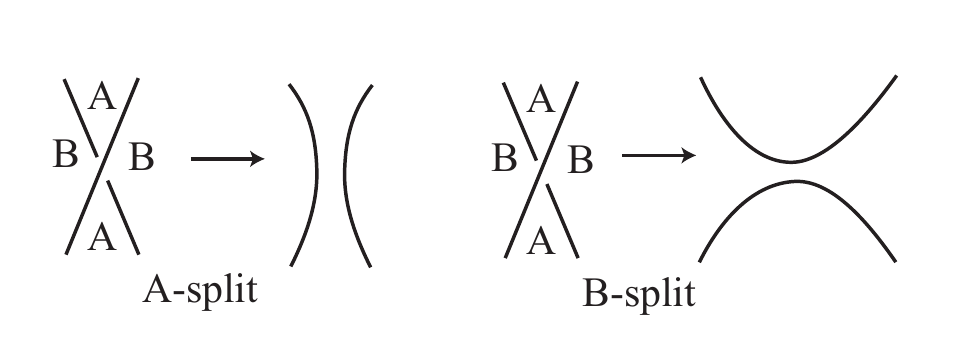}}
	\caption{A and B splittings in the construction of a layered surface.}\label{Fi:splits}
\end{center}
\end{figure} 

\begin{prop}

Suppose that the construction of a layered surface $S$ involves $a$ A-splits and $b$ B-splits.  Then $\tau(S,P)=a-b$.  It then follows that $l(S,L)=a-b+w(P)-2lk(L)$.

\end{prop}

\begin{proof}

Recall that $\tau(S,P)$ can be calculated just by counting the net total of positive and negative crossings between $L$ and $S \cap \partial N(L)$ through annuli in $S \cap N(L)$.  If $S$ is layered, we may have some pairs of crossings that appear as in Figure \ref{Fi:bend}, but each pair will have no net effect on $\tau(S,P)$.

\begin{figure}
\begin{center}
	\scalebox{.7}{\includegraphics{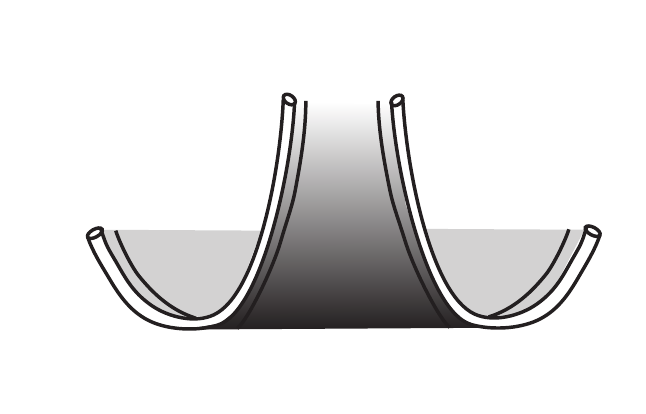}}
	\caption{The positive and negative crossings between $L$ and the parallel curve on $S$ cancel each other out wherever $S$ takes this form.}\label{Fi:bend}
\end{center}
\end{figure} 

Since $S$ is layered, the only other crossings between $L$ and $S \cap (\partial N(L))$ through annuli in $S \cap N(L)$ will occur at crossings (exactly why this is true will become more apparent when we introduce the Menasco Structure, although it should seem believable from our pictures so far).  The neighborhood of each A-split crossing in $P(L)$ will contain two positive crossings between $L$ and $S \cap (\partial N(L))$ through an annulus in $S \cap N(L)$, and the neighborhood of each B-split crossing will contain two negative crossings of this type.  Thus, $\tau=\frac{2a-2b}{2}=a-b$.  That $l(S,L)=a-b+w(P)-lk(L)$ then follows from the definition of $\tau(S,P)$.
\end{proof}

Note that a  layered surface constructed from an $n$-crossing projection in which the crossing splits produce $f$ circles will have an Euler Characteristic of $f-n$.  

Suppose we want to find a minimal genus layered surface for a given projection.  We could compare the genera of all $2^n$ possible surfaces, but we would like to do better than this.  Instead, we can employ the following:

\vspace {1 in}

 \underline{Minimal Genus Algorithm}

(1) Find the smallest $m$ for which the projection contains a projection $m$-gon.  

(2a) If $m \leq 2$, choose a projection $m$-gon and split each crossing on this region's boundary so that this region becomes a circle.  

(2b) If $m>2$ (and therefore $m=3$, by a simple Euler Characteristic argument on the projection plane), choose a projection $3$-gon.  From here, create two branches of our algorithm, each of which will ultimately yield a layered surface.  We will later choose to follow the branch that produced the smaller genus surface and to ignore the other.  For one of these branches, we split each crossing on this projection $3$-gon's boundary so that it becomes a circle.  For the other branch, we split each of these three crossings the opposite way.  

(3) Repeat until each branch reaches a projection without crossings.  Of all constructed surfaces, choose the one with the smallest genus.

\begin{theorem}

The Minimal Genus Algorithm always generates a minimal genus layered surface spanning a given alternating link.

\end{theorem}

Note, however,  that the resulting surface is not necessarily uniquely determined by this algorithm.  

\begin{proof}

(By induction on crossings in $P$.) The statement is trivial for zero crossings.  Suppose it is true for fewer than $n$ crossings, and let $P$ be an $n$-crossing projection.  Suppose that performing this algorithm on $P$ generates a layered surface $S$ spanning $L$.  Also, let $T$ be a minimal genus layered surface spanning $L$; then $g(T) \leq g(U)$ for any layered surface $U$ spanning $L$.  We want to show that $g(S) \leq g(T)$, or equivalently that $f(S) \geq f(T)$, where $f$ is the number of circles used in the construction of each respective layered surface.

First suppose that $P$ contains a projection $1$-gon or $2$-gon.  Let $G$ be the one around which all crossings are split in the construction of $S$, via our algorithm.  Also, suppose for contradiction that $f(S) < f(T)$.  We construct a new layered surface, $T'$, with the exact same crossing splits as in the construction of $T$, except that the crossing splits around $G$ coincide with the corresponding crossing splits for $S$.  Note that if the crossing splits of $S$ and $T$ already coincide here, we will have $T = T'$.  Figure \ref{Fi:bigon} demonstrates that $f(T') \geq f(T)$. Note that the first, third and fourth cases depicted cause a decrease in the number of circles if the crossing splits are switched, while in the second case, the number of circles may either decrease or be preserved.

We then have an $(n-1)$- or $(n-2)$-crossing projection where our algorithm yields $f(S)$ circles and where the construction of $T'$, another layered surface, yields a projection of $f(T') \geq f(T)>f(S)$ circles.  This contradicts our inductive hypothesis.

\begin{figure}
\begin{center}
	\scalebox{.5}{\includegraphics{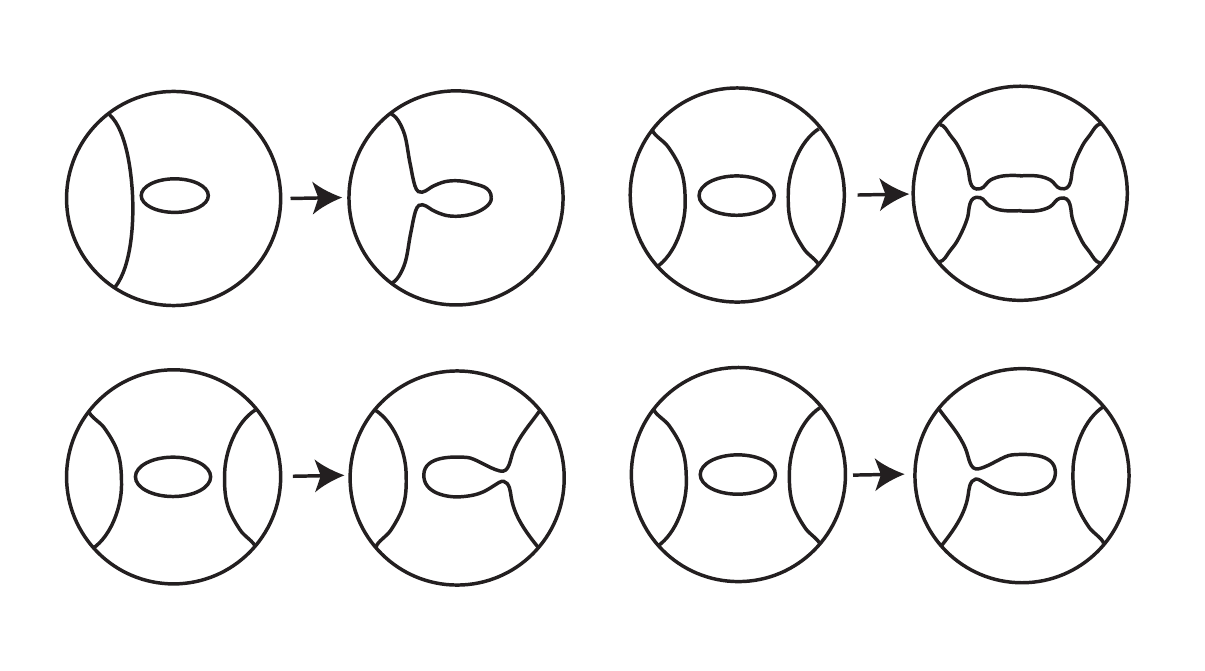}}
	\caption{Changing the crossing splits around a split $1$-gon or split $2$-gon  in our algorithm cannot increase the number of circles.} \label{Fi:bigon}
\end{center}
\end{figure}

Assume then that $P$ contains no projection $1$-gons or $2$-gons.  Let $G$ be the projection $3$-gon around which we first split crossings in the construction of $S$.  Let $P_1$ be the projection obtained from $P$ by splitting the crossings around $G$ so that $G$ becomes a circle.   Let $P_2$ be the projection obtained from $P$ by splitting each of these three crossings the opposite way.  $P_1$ and $P_2$ both have $n-3$ crossings.  Let $S_1$ and $S_2$ be the layered surfaces our algorithm generates from these respective projections.  Note that $f(S)=max\{f(S_1),f(S_2)\}$.  Let $T_1$ and $T_2$ be layered surfaces generated from $P_1$ and $P_2$, respectively, with crossing splits agreeing with those of $T$.  Our inductive hypothesis implies that $f(S_1) \geq f(T_1)$ and $f(S_2) \geq f(T_2)$. Figure \ref{Fi:triangle} shows that either $f(T_1) \geq f(T)$ or $f(T_2) \geq f(T)$.  It follows that $f(S) \geq f(T)$. 

\begin{figure}
\begin{center}
\scalebox{.5}{\includegraphics{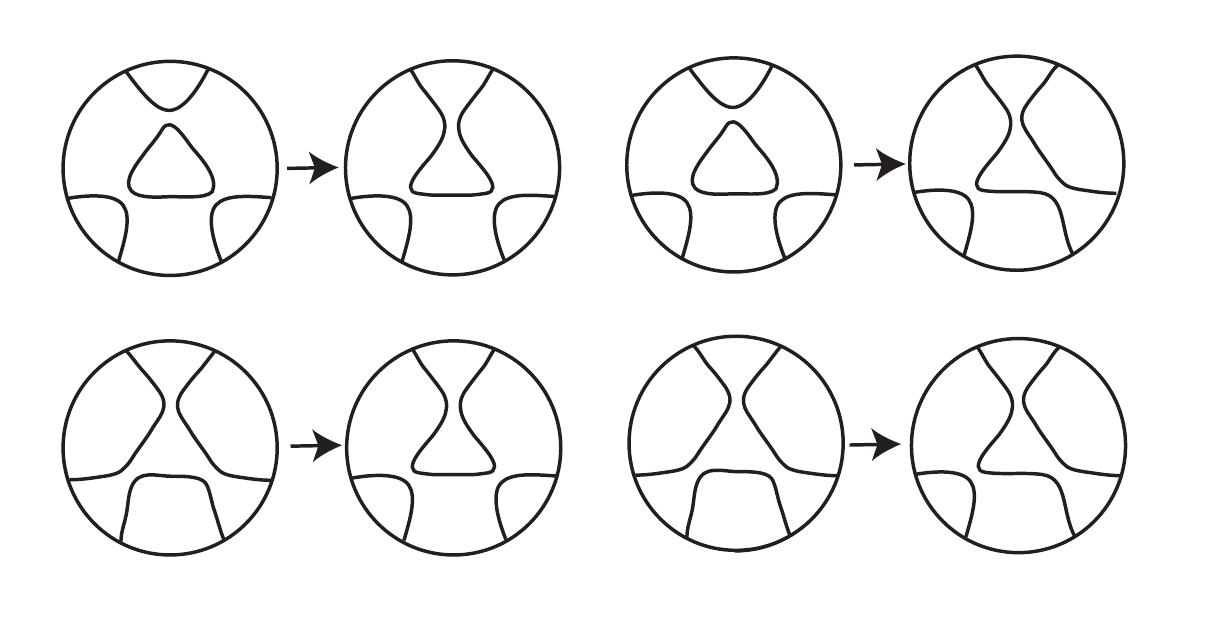}}
	\caption{Changing the crossing splits around a split $3$-gon from our algorithm cannot increase the number of circles for \emph{both} possible initial splittings. }\label{Fi:triangle}
\end{center}
\end{figure} 

Therefore our algorithm generates a minimal genus layered surface.  Theorem \ref{main} will imply that this surface is in fact a minimal genus spanning surface (layered or non-layered) for our link.  
\end{proof}

For Theorem \ref{main}, we extend the class of allowed surfaces in  
a manner that increases genus and perhaps changes boundary slope.  To do  
this, we allow the addition of \textbf{crosscaps} and \textbf 
{handles} to a basic layered surface.  These additions can be thought of as infinitesimal, local processes that take place in some isolated, unimportant (away from the crossings) part of the surface along the boundary.  These processes are depicted in Figure \ref{Fi:crosscap}.  

\begin{figure}
\begin{center}
	\scalebox{.5}{\includegraphics{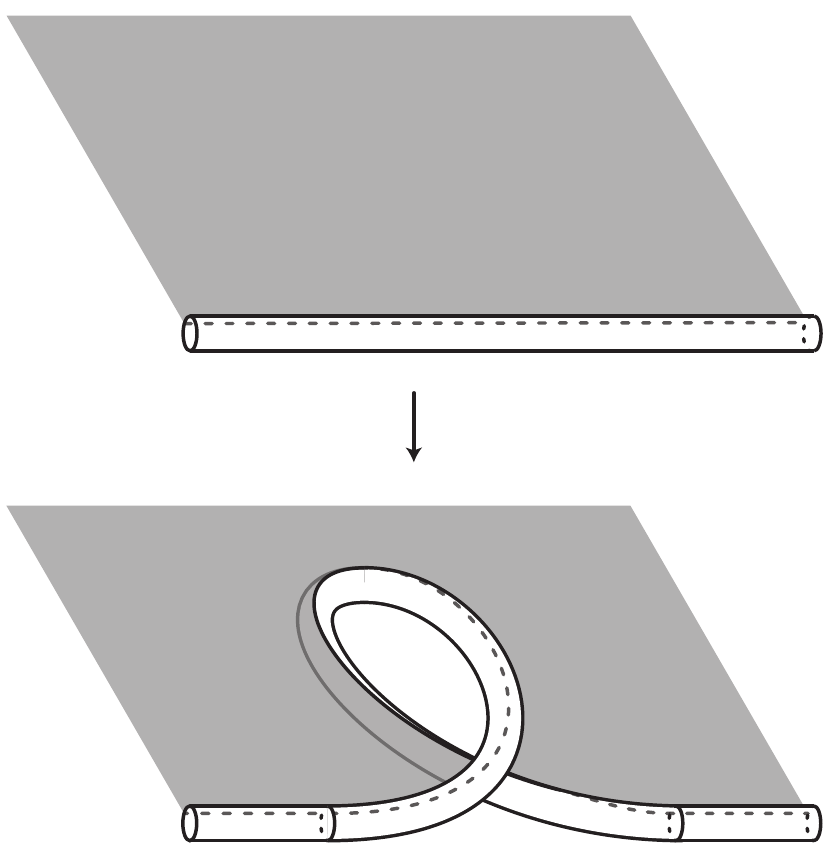}}
	\scalebox{.5}{\includegraphics{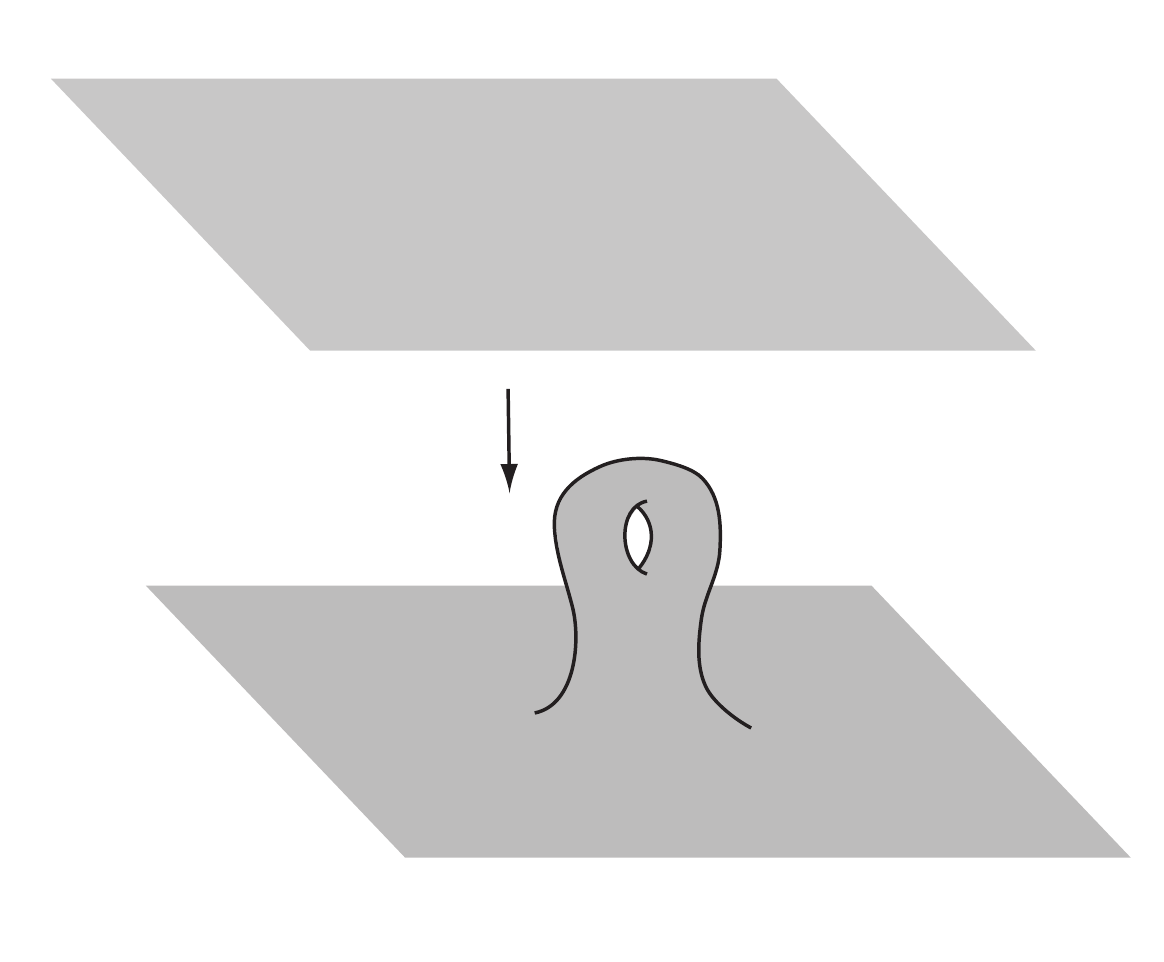}}
	\caption{Adding a crossscap (left) and adding a handle (right).}\label{Fi:crosscap}
\end{center}
\end{figure}

Neither adding a crosscap nor adding a handle need alter the projection $P(L)$. We can see this in the picture above for the handle, and in the later picture (see Figure \ref{Fi:crosscap2}) of adding crosscaps to a surface in Menasco form.  Adding a crosscap will either increase or decrease the twist (and therefore the aggregate slope) by two and will decrease the Euler Characteristic by one.  Adding a handle will decrease the Euler Characteristic by two and will not change twist or slope.  Adding a crosscap or a handle will necessarily produce an inessential surface. 

Note that a minimal nonorientable genus spanning surface need not be essential. It can be the case that the Minimal Genus Algorithm generates only orientable surfaces, in which case the minimal nonorientable genus surface is obtained from one of these by adding a crosscap, and the nonorientable genus is 1/2 greater than the orientable genus. (The knot $7_4$ is an example of this.) But this is as much greater than the orientable genus as the nonorientable genus can ever be.

 We are now ready to develop the primary piece of machinery that we use toward the proof of  Theorem \ref{main}.

\section{Menasco Structure}

We utilize William Menasco's geometric  
machinery for analyzing alternating links as first appeared in \cite{M1}. A variety of results have been proved using this technique, including the fact that any two reduced alternating projections of a given link are related through flyping(\cite{M4} and \cite{M6}), an alternating link is splittable if and only if its alternating projections are disconnected, and an alternating knot is composite if and only if it is obviously composite in any reduced alternating projection( \cite{M1}). Additional results for alternating links and extensions of alternating links have also been proved(e.g. \cite{A1}, \cite{A2}, \cite{A3}, \cite{A4}, \cite{C}, \cite{M4}, \cite{M5}, \cite{Ts}). We will use the Menasco machinery to prove that a spanning surface must have an arc whose neighborhood in $S$ can be isotoped into a``crossing band". This will allow an inductive argument to follow.)

\subsection {Background}

 Begin with a reduced, alternating  
projection on a sphere.  From this projection, create an embedding of  
$L$ that lies on a sphere, $S^2 \subset S^3$, except in the neighborhood of each crossing, where we insert a ball so that the over-strand can run over this ball and the under-strand can run under it.  We will call every such crossing ball a \textbf{Menasco ball}.  This will give rise to a picture as in Figure \ref{Fi:menasco} on the left  and a view from above as on the right.  We will assume that $L$ is not splittable.  Otherwise, it would appear obviously so in the projection (see \cite{M1}) and we could consider each non-splittable component separately.

\begin{figure}
\begin{center}
	\scalebox{.5}{\includegraphics{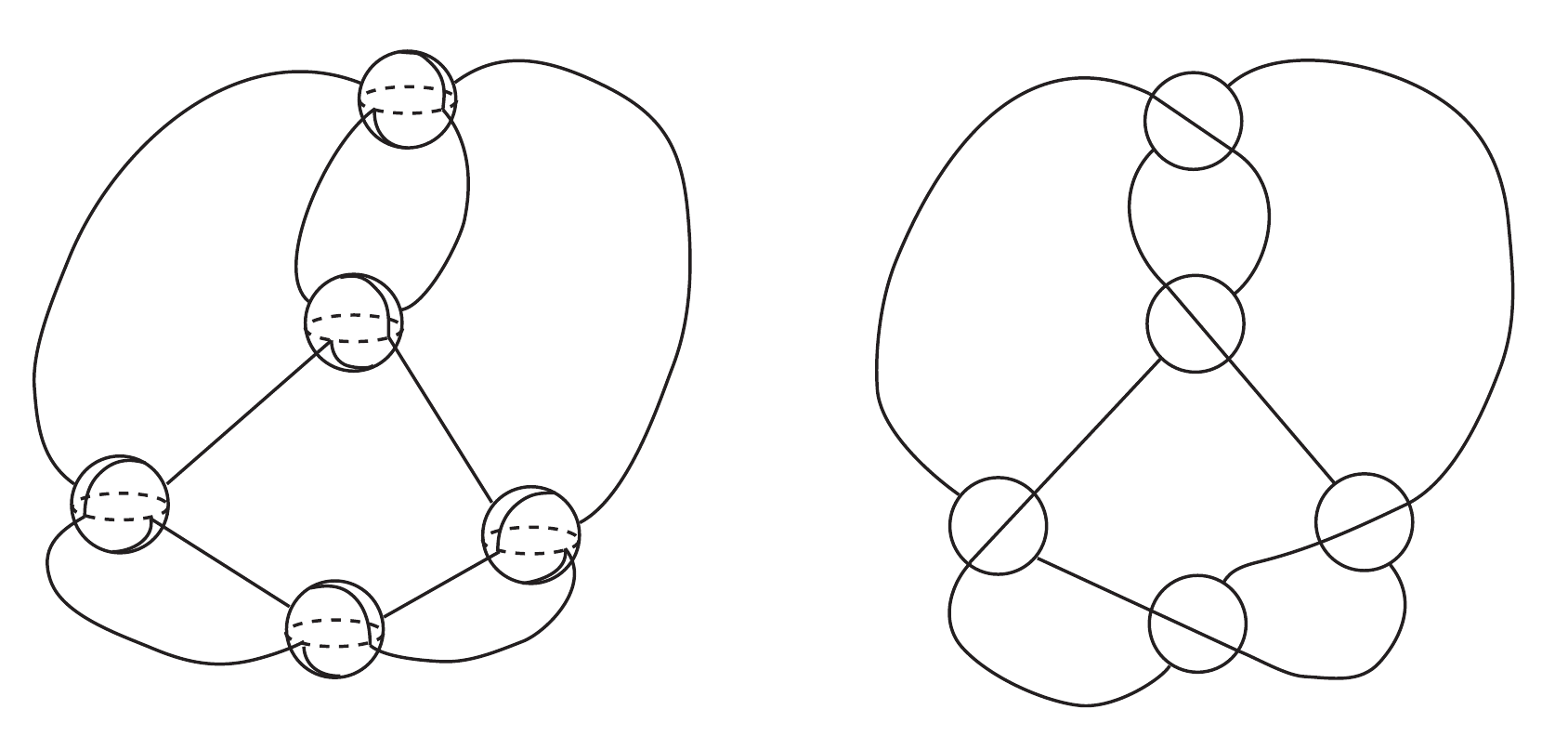}}
	\caption{A embedding of the $5_2$ link in Menasco form.}\label{Fi:menasco}
\end{center}
\end{figure} 

We will call such an embedding of $L$ relative to $S^2$ a \textbf 
{Menasco projection} $P$ with $n$ crossings.  We denote the union of the collection of Menasco balls by $M=\cup_{i=1}^{n} M_i$ and we let $F=S^2 \backslash $\textit{\r{M}}.

Now suppose we have an incompressible (but not necessarily boundary-incompressible) surface $S$ spanning $L$.  By general position, we can isotope $S$ so that $S \cap F$ is a collection of  simple closed curves and arcs with each endpoint either on $L$ or on the equator of a Menasco ball.  We can also choose this isotopy so that $S$ is disjoint from the interiors of the Menasco balls, wherever possible, including everywhere along $N(L)$.  That is, $S \cap N(L) \cap M = \varnothing$. Note that later, we will qualify this requirement to allow intersections of $S$ with $N(L) \cap M$ called ``crossing bands''. 

 We assume that $N(L)$ is small relative to the Menasco balls. Utilizing the incompressibility of $S$, we can isotope $S$ to eliminate all simple closed curves in $S \cap F$. 

Menasco shows that once we isotope $S$ in this way, anywhere it is \emph{still forced} to intersect the interior of a Menasco ball it must do so in a saddle as in Figure \ref{Fi:saddle}. 

\begin{figure}
\begin{center}
	\scalebox{.8}{\includegraphics{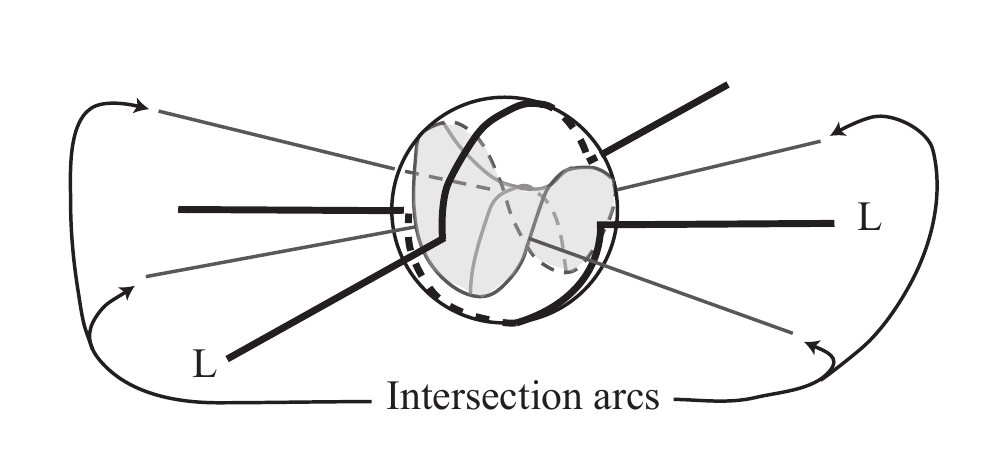}}
	\caption{The surface intersects each Menasco ball in a possibly empty collection of saddles. }\label{Fi:saddle}
\end{center}
\end{figure} 

 We define an \textbf{intersection arc} to be an arc of $S \cap F$, with each endpoint lying either on the link or on the equator of a Menasco ball.  Each intersection arc lies entirely  
within a single connected region in $F \backslash P$. 

Then, for each Menasco ball $M_i$, each of the four regions of $F \backslash P$ adjacent to $M_i$ will have the same number of endpoints of intersection arcs on $\partial M_i$, that number being the number of saddles. 
Define $F_+$ to be the union of $F$ with the upper hemispheres of the Menasco balls and $F_-$ to be the union of $F$ with the lower hemispheres. Then at a Menasco ball with a single saddle, the intersection curves of $S$ with $F_+$ and $F_-$ appear as in Figure \ref{Fi:menview}. 

\begin{figure}
\begin{center}
	\scalebox{.5}{\includegraphics{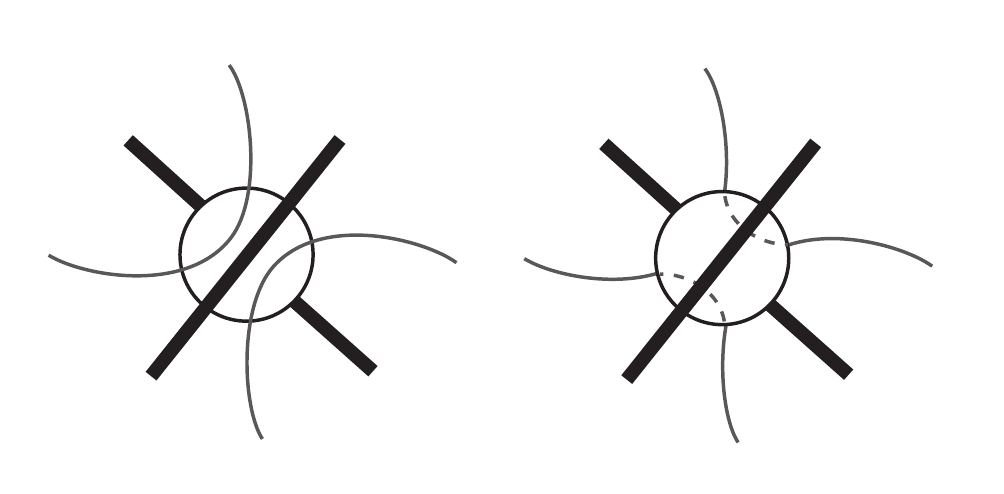}}
	\caption{Pictures of the arcs of intersection between $S$ and $F_+$, and between $S$ and $F_-$.}\label{Fi:menview}
\end{center}
\end{figure}  

Since $S$ is incompressible, we then see that $F \cup \partial M$ cuts  $S$ up into a collection of disks, some inside the sphere, some  
outside the sphere, and some inside Menasco balls.  We will call these \textbf{under-disks}, \textbf{over-disks}, and \textbf{saddles}, respectively. Up to isotopy, $S$ is then uniquely determined by the way it intersects $F$.   A final important fact from \cite{M1} is that no intersection arc has both endpoints on $\partial M_i$ for any Menasco ball $M_i$.  Otherwise, it could easily be removed.

We call an incompressible surface that has been isotoped in this manner to be \emph{in Menasco form}.
Figure \ref{Fi:mensurface}  depicts a layered surface for the $5_2$ knot that has been isotoped to be in Menasco form. We see $F_+$ with the over-disks shaded and the under-disks labelled $U_i$. Note that all layered surfaces can be isotoped to be in Menasco form without any saddles.

\begin{figure}
\begin{center}
	\scalebox{.6}{\includegraphics{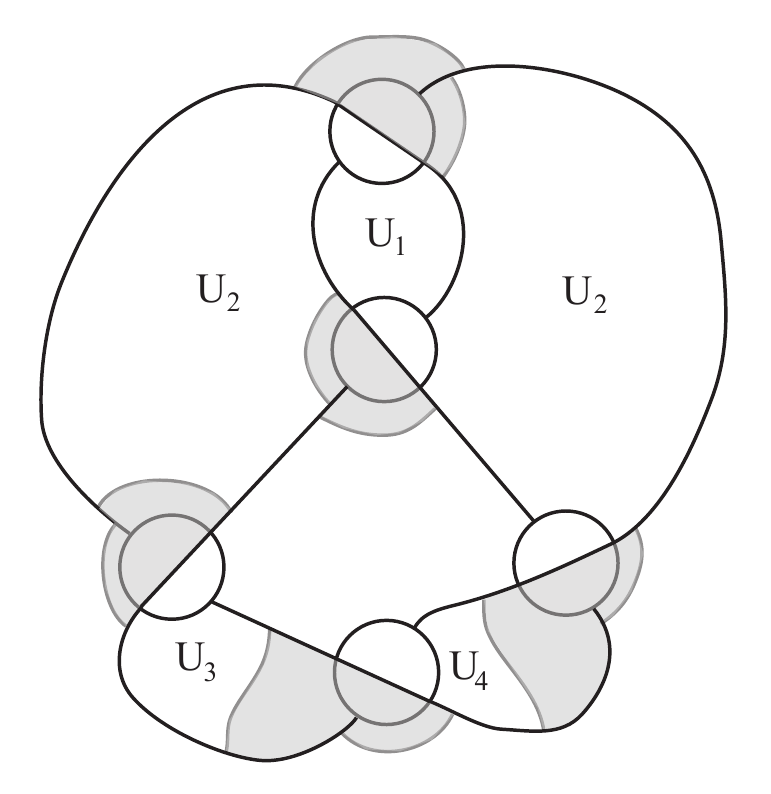}}
	\caption{The view of $F_+$ depicting a layered surface spanning the $5_2$ knot, isotoped to lie in Menasco form, with over-disks shaded and under-disks labelled.}\label{Fi:mensurface}
\end{center}
\end{figure}

\subsection{Cleaning}

In this section, we describe some means to simplify surfaces in Menasco form.  In general, we want to minimize the number of intersection arcs between $S$ and $F$, so several of our results will describe ways of doing this.  This will generally involve isotoping $S$ or altering it in some other well-defined way to eliminate particular types of arcs of intersection, or else showing that certain types of arcs of intersection simply cannot occur.  As before, we call our spanning surface $S$, our link $L$, our Menasco projection $P$ with $n$ crossings, our projection surface $S^2$, and our collection of Menasco balls $M$.  Also, as before, we let $F=S^2 \backslash $\textit{\r{M}} be our projection surface. 

Given an intersection arc, there exists an over-disk and an under-disk, each of which contains that arc on its boundary.  Additionally, each point on $L$ lies on the boundary of either an over-disk  or an under-disk, with each intersection arc endpoint on $L$ marking the transition between the boundary of an over-disk  and the boundary of an under-disk. Each intersection arc endpoint on a  Menasco ball boundary also lies on the boundary of a saddle.  Every point on the boundary of an over-disk or under-disk lies on an intersection arc, on the boundary of a Menasco ball, or on $L$.  For the proof of Theorem \ref{main}, we heavily utilize this alternation between over-disks and under-disks.

 In accordance with this alternation of over-disk  and under-disk  boundaries on $L$, where `+' denotes a portion of the boundary of an over-disk  and `$-$' denotes a portion of the boundary of an under-disk, and  transitions occurring exactly at the endpoints of intersection arcs, we can see that every intersection arc with endpoints on $L$  will take one of the forms shown in Figure \ref{Fi:arcs}.

\begin{figure}
\begin{center}
	\scalebox{.5}{\includegraphics{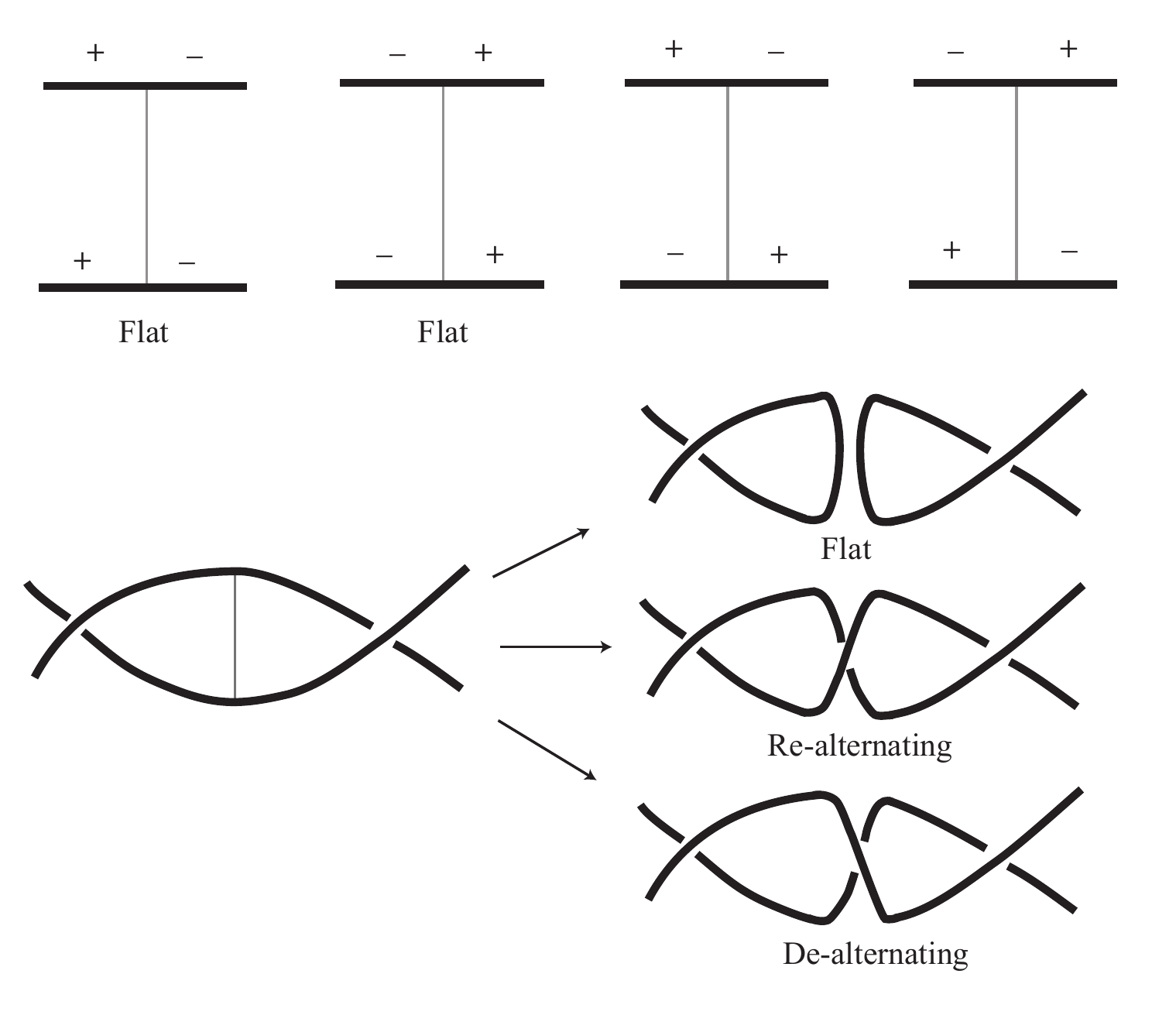}}
	\caption{Four possibilities for intersection arcs and the possible results from cutting along an arc. }\label{Fi:arcs}
\end{center}
\end{figure}  

If we imagine cutting $S$ along any of the four types of arc, by removing a neighborhood of this arc from $S$,  a new surface $S'$ is generated,  bounding a new link $L'$, which will inherit a projection $P'$.  If we cut along either of the first two types of arcs, $P'$ will still be alternating with $n$ crossings.  We call these first two \textbf{flat} intersection arcs.  If we cut along either of the latter two types of arcs, $P'$ will have $n+1$ crossings.  For one of these cases, $P'$ will no longer be alternating, while for the other it will be.  We will call the former arc a \textbf{de-alternating} intersection arc and the latter a \textbf{re-alternating} intersection arc.  Every intersection arc with both its endpoints on $L$ will be flat, de-alternating, or re-alternating, as in Figure \ref{Fi:arcs}.

Among all embeddings of $S$ relative to $F \cup M$ that preserve the  
already established structure (Menasco projection, general position, $S$ disjoint from $M$ wherever possible), we want to choose the \emph{best} one.  In particular, we want to \emph{lexicographically minimize} the number of intersection arcs and saddle-disks.  That is, we want to minimize the number of intersection arcs, and once we have done this we want to minimize the number of saddle-disks as well.  If an embedding of $S$ relative to $F \cup M$ does this, we will say that it is \textbf{clean}. 

\begin{prop} \label{Prop:clean0}
In a clean embedding, no adjacent pair of intersection arc endpoints on $P \cap F$ will connect to intersection arcs that both lie in the same region of $F \backslash P$, as in Figure \ref{Fi:cleaning1}.
\end{prop}

\begin{figure}
\begin{center}
	\scalebox{.5}{\includegraphics{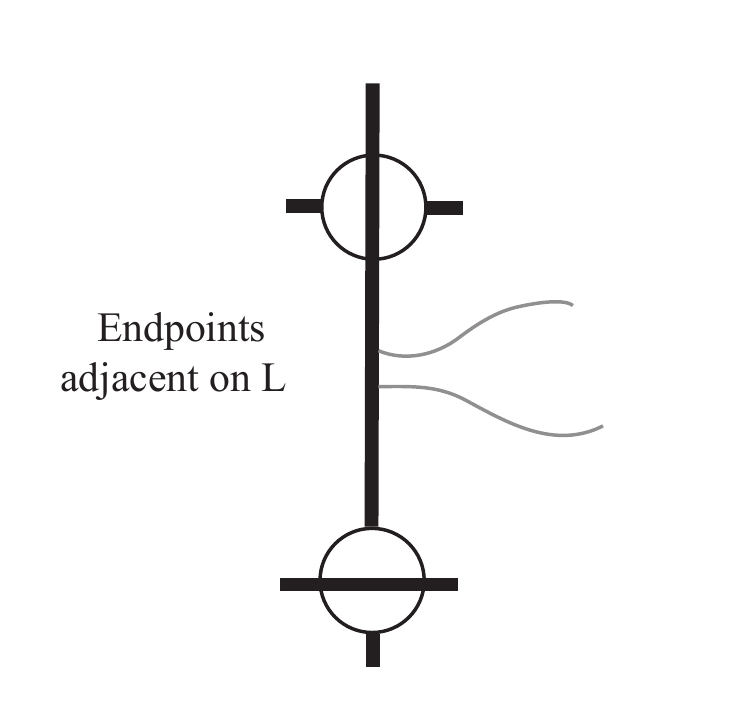}}
	\caption{No clean embedding of a surface in Menasco form contains this conformation.}\label{Fi:cleaning1}
\end{center}
\end{figure}  

\begin{proof}
Suppose our clean embedding contains such a pair of adjacent intersection arcs, as in \ref{Fi:cleaning1}.  The portion of $P \cap F$ connecting the two intersection arc endpoints lies on the boundary of either an over-disk  or an under-disk.    We can isotope this disk  through $F$ along the portion of $P \cap F$ that connects the two intersection arc endpoints.  Doing this will decrease the number of intersection arcs, contradicting the fact that our Menasco structure was clean.
\end{proof}

\begin{prop} \label{Prop:clean2}       
Suppose  $D$ is an over (under)-disk  whose boundary contains intersection arcs $I$ and $J$, both in some region $R \subset F \backslash P$ where $I$ and $J$ can be connected by an arc $\alpha \subset F \backslash (P \cup S)$.  Then, we can push a particular part of $D$ through $S^2 \cap N(\alpha)$ so that it is divided into two separate over (under)-disks.  If our original embedding of $S$ was clean, then this will be as well. 
\end{prop}

\begin{proof}
See Figure \ref{Fi:cleaning2}. Suppose we have such a $D$, $R$, $I$, $J$, and $\alpha$.  Let $\beta \subset D$ be an arc with the same endpoints as $\alpha$.  Then $\alpha \cup \beta$ bounds a disk  $X$ such that $N(\partial X) \cap S  \cap$ \textit{\r{X}}$= \varnothing$ and $L \cap$  \textit{\r{X}} $= \varnothing$.  Therefore \textit{\r{X}}$\cap S$ will consist entirely of simple, closed curves of intersection, each of which can be removed because $S$ is incompressible.  Therefore, we can assume that \textit{\r{X}}$\cap S = \varnothing$.  We can then isotope $N(\beta)$ through $S^2 \cap N(\alpha)$ via $N(X)$, so that $D$ is divided into two separate over-disks, as desired.

Since the number of intersection arc endpoints does not change, the number of intersection arcs will not change.  Also, $S$ was originally in Menasco form, and we have done nothing to change this.  If the original embedding was clean, this will be as well.  
\end{proof}

\begin{figure}
\begin{center}
	\scalebox{.5}{\includegraphics{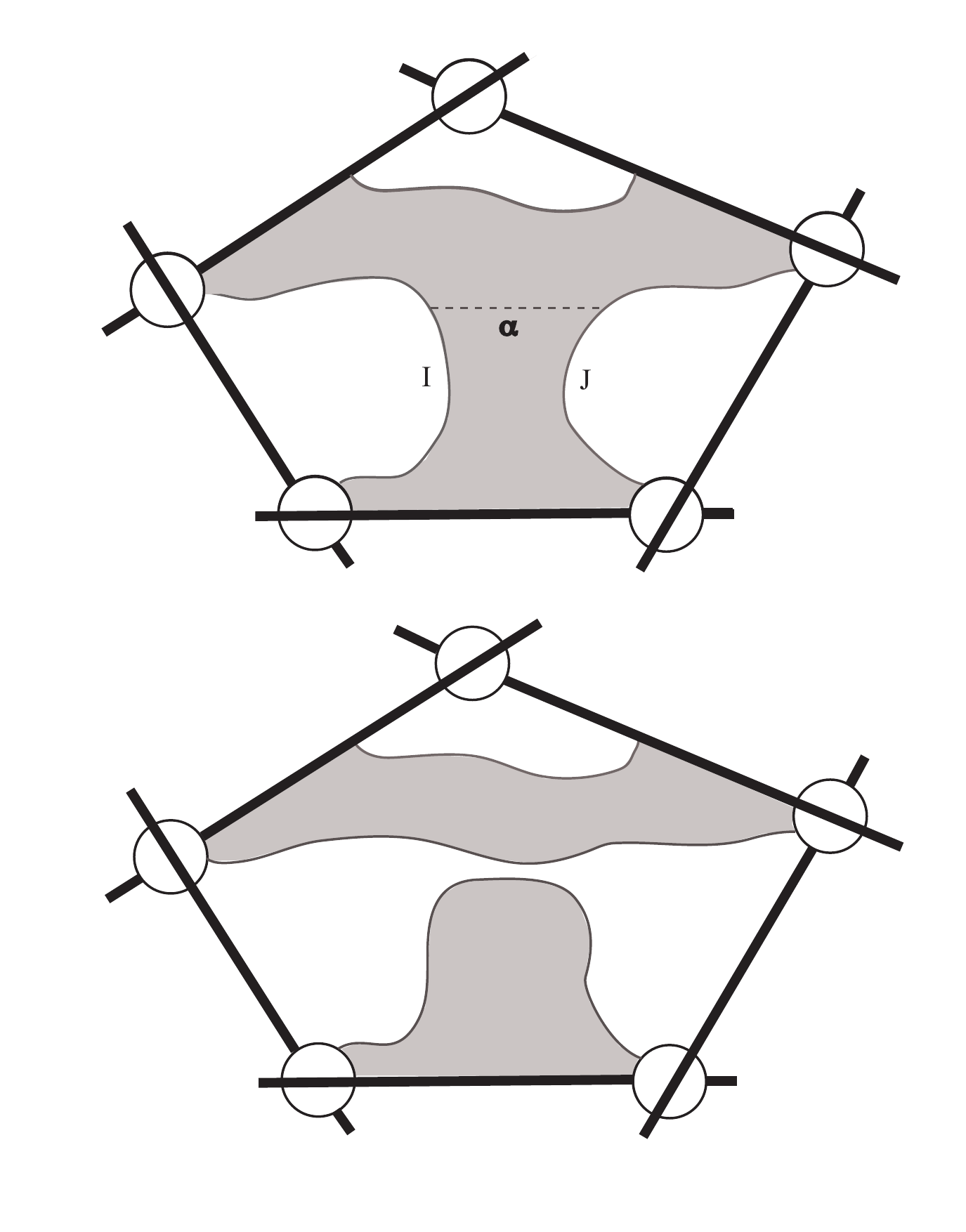}}
	\caption{Pushing the disk  $D$ through the projection plane.}\label{Fi:cleaning2}
\end{center}
\end{figure}  

Note that this move is \emph{not} a simplification, but rather a convenient freedom.  This convenience will soon become manifest.

\begin{prop}  \label{prop:cleaner}      

Suppose $S$ is in a clean embedding, and $D$ is either an over-disk  or an under-disk  whose boundary contains intersection arc $I$ in some region $R \subset F \backslash P$ where an endpoint of $I$ lies on a Menasco ball $M_i$.  Then $\partial D$ will not contain either portion of $L$ that traverses $M_i$, as in the top of Figure \ref{Fi:cleaning4}.

\end{prop}

\begin{proof}

Suppose, without loss of generality, that $D$ is an over-disk .  Clearly $\partial D$ will not contain the under-pass of $M_i$.  Suppose it contains the over-pass.  The bottom of Figure \ref{Fi:cleaning4} shows the isotopy that removes the corresponding saddle-disk  at $M_i$ without increasing the number of intersection arcs, contradicting the fact that our embedding of $S$ was clean.
\end{proof}

\begin{figure}
\begin{center}
	\scalebox{.5}{\includegraphics{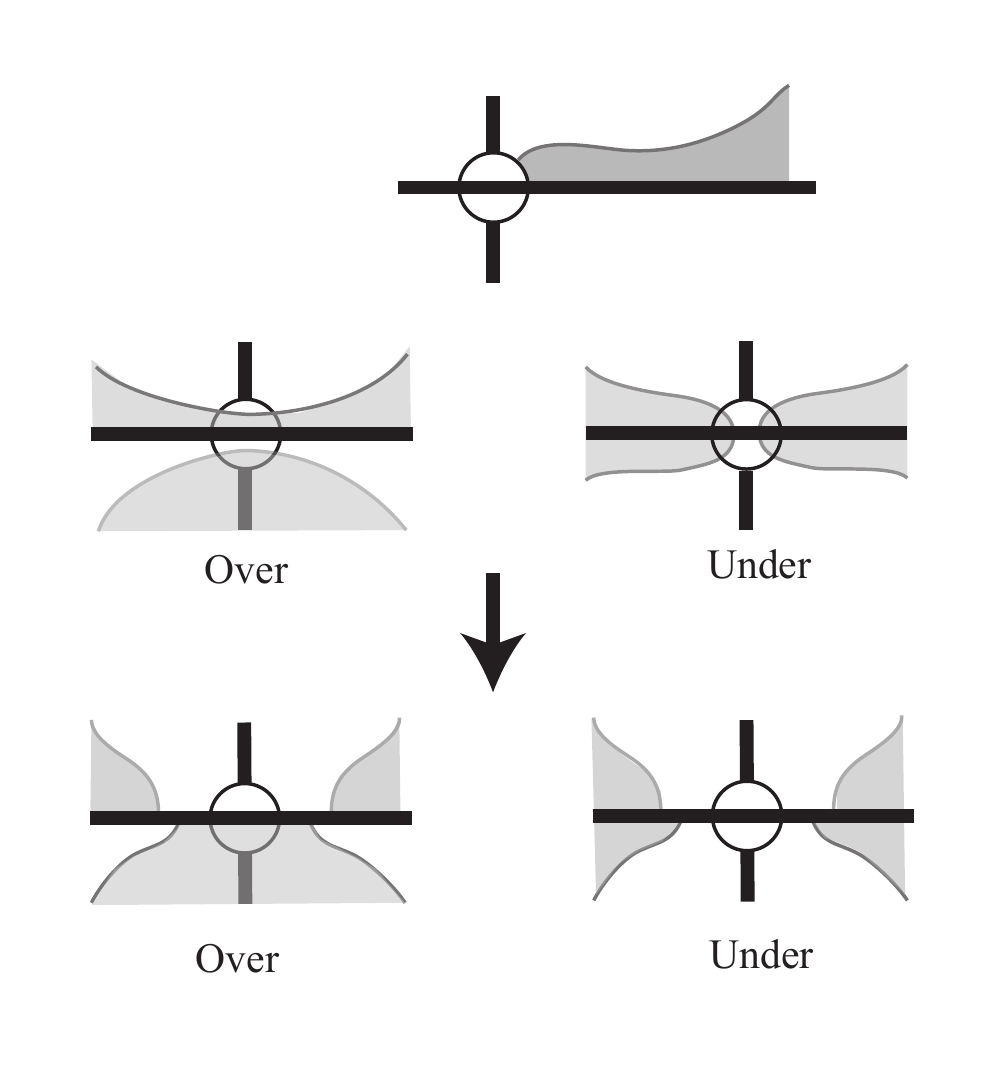}}
	\caption{Another removable structure in our cleaning process.}\label{Fi:cleaning4}
	\end{center}
\end{figure} 

The next two lemmas assume that the spanning surface in question is also meridianally boundary-incompressible. 

\begin{prop} \label{prop:clean}
If $S$ is meridianally boundary-incompressible in a clean embedding, then no intersection arc has both endpoints on the same component of $P \cap F$.
\end{prop}

\begin{proof}
 Suppose we have such an intersection arc, $I$.  See the two conformations in Figure \ref{Fi:cleaning9}. The region of $F$ bounded by $I$ together with the portion of $P \cap F$ that connects the endpoints of $I$ may or may not contain additional arcs of intersection in its interior.  If it does, choose an innermost one, $I^*$.  Now the region of $F$ bounded by $I^*$ together with the portion of $P \cap F$ that connects the endpoints of $I^*$ does not intersect $S$ in its interior.  This region of $F$ is a disk  whose boundary lies entirely on $P$, except along $I^* \subset S$ and that does not intersect $S$ in its interior.  Therefore, either this disk represents a meridianal boundary compression or $I^*$ cuts a disk  off of $S$ and can thus be pushed into $L$, reducing the number of intersection arcs.  The former possibility contradicts our assumption that $S$ is meridianally boundary-incompressible, and the latter contradicts the assumption that our embedding of $S$ is clean.  Therefore, no such $I$ can exist.
\end{proof}

\begin{figure}
\begin{center}
	\scalebox{.5}{\includegraphics{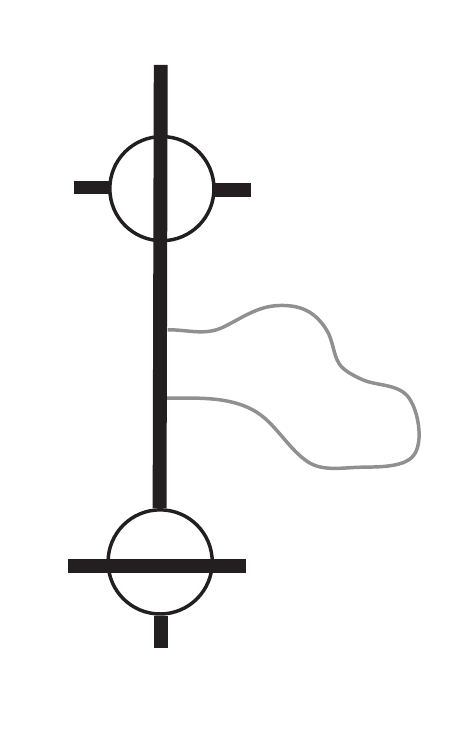}}
	\caption{An additional removable structure in our cleaning process.}\label{Fi:cleaning9}
	\end{center}
\end{figure}

\begin{prop}  \label{Prop:clean3}      

Suppose $S$ is meridianally boundary-incompressible.  A clean embedding will contain no intersection arc $I$ with one endpoint on a Menasco ball boundary $\partial M_i$ and the other on a portion of $P \cap F$ that connects to $M_i$, as in Figure \ref{Fi:cleaning5}.

\end{prop}

\begin{figure}
\begin{center}
	\scalebox{.5}{\includegraphics{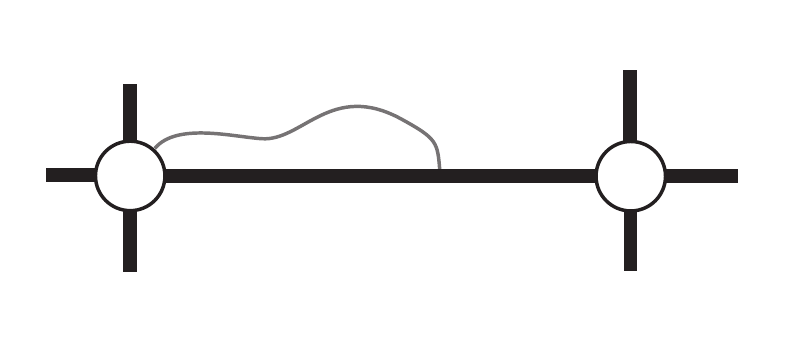}}
	\caption{A final conformation for our cleaning process to remove.}\label{Fi:cleaning5}
\end{center}
\end{figure}  

\begin{proof}

Suppose we do have such an intersection arc $I$.  Consider the region bounded by $I$, $\partial M_i$, and the portion of $P \cap F$ that contains an endpoint of $I$.   Notice that any intersection arc in this region must have one endpoint on $\partial M_i$ and the other on $P$, since no intersection arc can have both endpoints on either of these, as shown earlier by Menasco and by Proposition \ref{prop:clean}.  Choose an innermost arc of intersection; that is, one whose analogously bounded region does not intersect $S$ on its interior.  Call this intersection arc $I^*$.  

Now consider the portion of $P \cap F$ connecting the endpoint of $I^*$ on $P$ with $\partial M_i$.  Its interior will either contain no intersection arc endpoints or exactly one endpoint, that being of an intersection arc that lies in the region of $F \backslash P$ \emph{adjacent} to the one containing $I$ and $I^*$.

First suppose that this portion of $P$ contains no intersection arc endpoints.  Then, we get an immediate contradiction by Proposition \ref{prop:cleaner}.

Now suppose that the aforementioned portion of $P$ contains one intersection arc endpoint.  This gives the picture in Figure \ref{Fi:cleaning7}, which Proposition  \ref{Prop:clean2} allows us to alter as shown.  We then have a disk  (depicted here as an over-disk) whose boundary lies entirely on the interior of $S$, except for a single arc along $P$.  Then, either this disk represents a meridianal boundary compression (contradicting the fact that $S$ is meridianally boundary-incompressible) or its boundary cuts a disk  off of $S$ and can therefore be pushed into $L$, which would decrease the number of intersection arcs, thereby contradicting the assumption that $S$ was embedded cleanly.
\end{proof}

\begin{figure}
\begin{center}
	\scalebox{.75}{\includegraphics{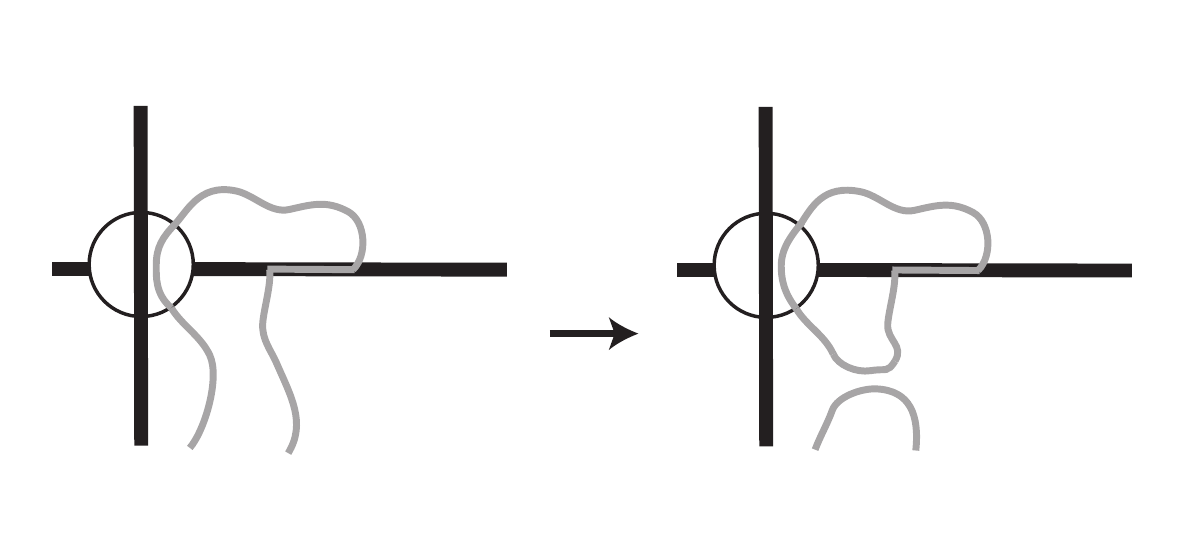}}
	\caption{Using Proposition \ref{Prop:clean2} to obtain a potential boundary-compression disk .}\label{Fi:cleaning7}
\end{center}
\end{figure}  

Also, now that we have the Menasco structure at our disposal, we can provide a different picture for what it means to add a crosscap.  One can add a crosscap to a surface in Menasco form by performing one of the two operations shown in Figure \ref{Fi:crosscap2}. This has the same impact as adding a crosscap as in Figure \ref{Fi:crosscap} and then pulling the knot taut.

\begin{figure}
\begin{center}
	\scalebox{.5}{\includegraphics{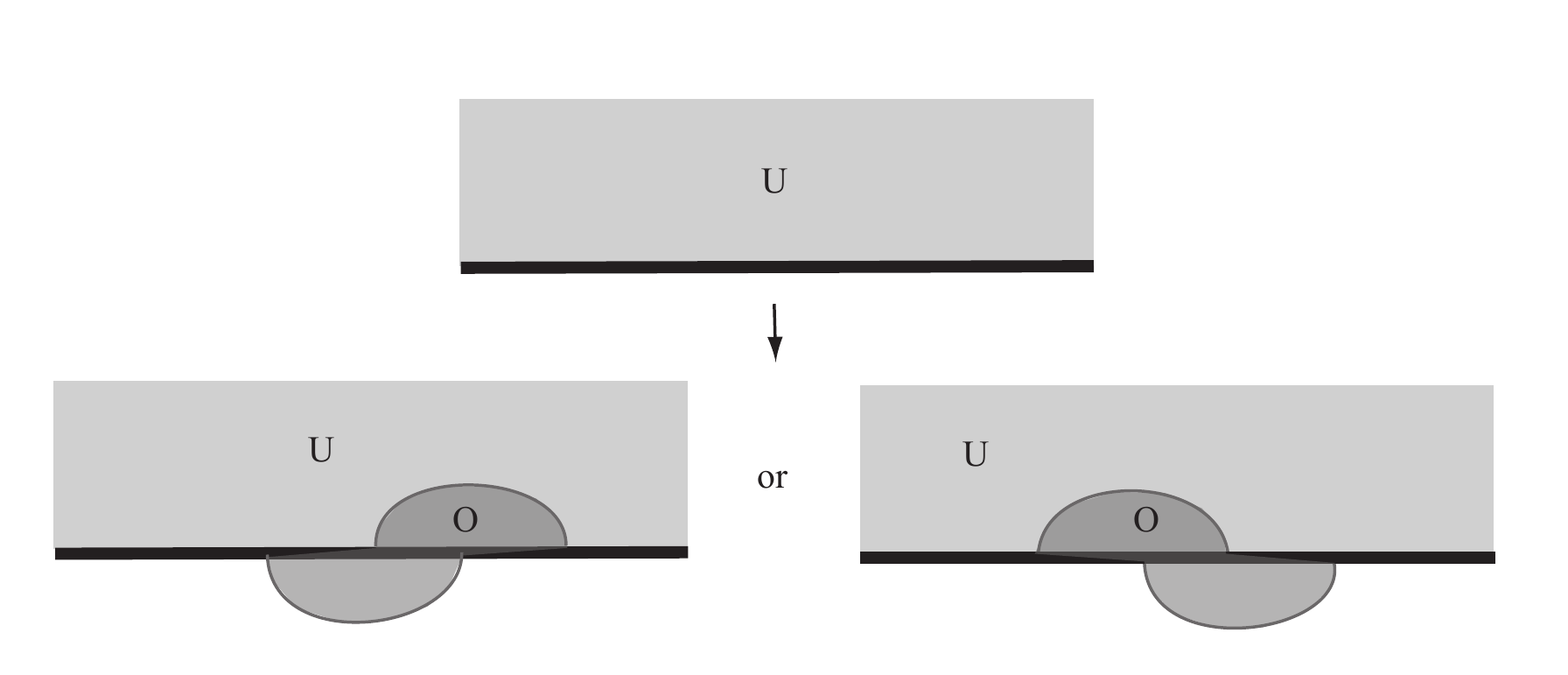}}
	\caption{Adding a crosscap to a surface in Menasco form corresponds to adding two intersection arcs as in this figure.}\label{Fi:crosscap2}
\end{center}
\end{figure}  

\subsection {Combinatorial Properties}

Suppose that we have a clean embedding of a meridianally essential surface $S$.  In this section, we will lay out several  combinatorial properties of $S \cap F$.  From these, the proof of  Theorem \ref{main} will follow somewhat naturally.

Suppose we incorporate into a Menasco projection $P$ a collection $A=\cup_{i=1}^{r} A_i$ of re-alternating intersection arcs, all of whose endpoints lie on $P$ (rather than on a Menasco ball).  Recall that because $S$ is meridianally boundary-incompressible none of these will have both endpoints on the same edge of a projection $n$-gon of $P$.  We define a \textbf{segment} of $P$ with respect to $A$ to be a portion of $P \cap F$  
that is disjoint on its interior from endpoints of arcs in $A 
$ and each of whose endpoints lies either on the boundary of a  
Menasco ball $M_i$ or at an endpoint of an arc in $A$.  See Figure \ref{Fi:segments}.  Note  
that if $A=\varnothing$ then each segment connects a pair of adjacent Menasco balls.

\begin{figure}
\begin{center}
	\scalebox{.9}{\includegraphics{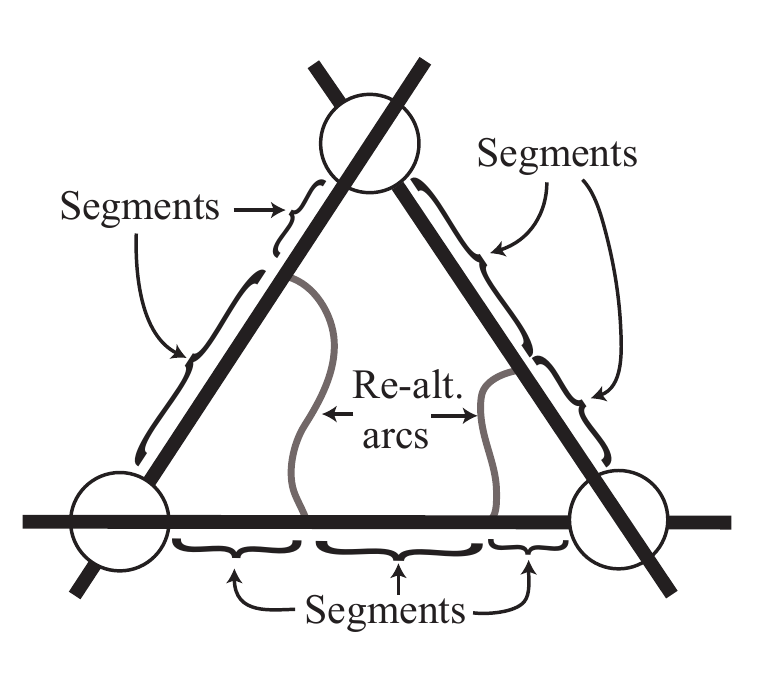}}
	\caption{Segments of $P$.}\label{Fi:segments}
\end{center}
\end{figure}  

Let $P^* = F \cap (P \cup A)$.  As each segment lies on the boundaries of exactly two regions of $F \backslash P^*$, we can associate two ``boundary segments''  
to any segment.  We thus define a \textbf{boundary segment} of $P$ with respect to $A$  
to be a pairing of a region $R$ of $F \backslash P^*$  
with a segment of $P$ (with respect to $A$) on its boundary.  As with a segment, given a boundary segment that lies on the boundary of a particular region, each endpoint of that boundary segment either will be an intersection arc endpoint on that region's boundary or will lie on the boundary of a Menasco ball that forms a portion of that region's boundary.

For the following remarks, suppose that our projection $P$ contains $n$ crossings, that $A$ contains $r$ re-alternating arcs, all of whose endpoints lie on $P$, and that $P^* =  F \cap (P \cup A)$, as before.

\bigskip

\begin{prop}

$P^*$ will contain  
exactly $2n+2r$ segments and $4n+4r$ boundary segments.  Each segment will lie on the boundary of an over-disk at one end and the boundary of an under-disk at the other. 
\end{prop}

\begin{proof}

First suppose that $r=0$; that is, $A=\varnothing$.  Each of the $n$ crossings (Menasco balls) provides an endpoint for four segments of $P$.  Since each segment has two endpoints, we must have $2n$ segments and thus $4n$ boundary segments.  Since $P$ is alternating, each segment must bound part of an over-disk  at one end and part of an under-disk  at the other.  The statement thus holds for $r=0$.  

Now, suppose we add the re-alternating arcs of $A$ one at a time.  Every time we add an arc, that arc has two endpoints, each on some segment of $P$.  Each endpoint divides its segment into two new segments, so each arc in $A$ increases the number of segments by two.  Thus, we end up with $2n+2r$ segments.  Because each arc in $A$ is re-alternating, each segment will still bound part of an over-disk  at one end and part of an under-disk  at the other.
\end{proof}

\begin{prop}
Each of the $2n+2r$ segments (and at least this many  
boundary segments) of $P^*$ contains an odd, and therefore positive,  
number of intersection arc endpoints on its interior.
\end{prop}

\begin{proof}

$P$ everywhere bounds either an over-disk  or an under-disk  and switches exactly where it contains an intersection arc endpoint.  Since each segment of $P^*$ bounds an over-disk  at one end and an under-disk  at the other, it must switch an odd number of times.  Therefore, it must contain an odd number of intersection arc endpoints.
\end{proof}

\begin{prop}\label{Prop:boundary}

$P^*$ divides $F$ into $n+r+2$ regions, each with at least two boundary segments on its boundary.

\end{prop}

\begin{proof}

First suppose $A=\varnothing$.  If $n=0$ we have two regions, and each time we add a crossing we add a region.  Thus, for arbitrary $n$ we have $n+2$ regions.  Since each arc in $A$ will separate a region in two, thus increasing the number of regions by one, we conclude that $P^*$ divides $F$ into $n+r+2$ regions.

Clearly the boundary of every region must contain at least one boundary segment.  Suppose the boundary of some region contains only one boundary segment. Then either the two endpoints occur at the same crossing, which contradicts the fact that $P$ is reduced, or both endpoints of a re-alternating arc in $A$ are on the same edge of a projection $n$-gon of $P$, contradicting the assumption that $S$ is meridianally essential and cleanly embedded, by Proposition \ref{prop:clean}.
\end{proof}

Given $P^*$ as above, suppose that $R$ is a region  
of $F \backslash P^*$.  Also, suppose that $R \cap S=B$, where $B$  
is a (possibly empty) collection of intersection arcs.  Consider $B$ as a subset of $R$.  
 This creates a new regional picture of  
$R \subset F$, which reveals how $S$ intersects $R$.

If $B$ is non-empty, there exist at least two outermost components of $R \backslash B$; that is, ones whose boundary contains exactly one arc from $B$.  Since no segment contains both endpoints of any intersection arc from $A$, the boundary of each outermost region contains portions of at least two segments of $P^*$.

\section {Main Results}

The proof of Theorem \ref{main} will proceed by induction on the number of crossings.  First, however, we use the results of the previous section to prove a critical lemma, which will be helpful in reducing the $n$-crossing case to an $(n-1)$-crossing case. 

We will say that a surface $S$ intersects a Menasco ball $M_i$ in a \textbf{crossing band} if $S \cap M_i$ consists of a disk bounded by the over-strand and under-strand on $\partial M_i$ together with a pair of opposite arcs on the equator of $M_i$, as in Figure \ref{Fi:desirednew} and Figure \ref{Fi:desired}.  The two possible configurations of a crossing band at a particular crossing correspond to the two pairs of opposite arcs on the Menasco ball's equator.  Note that a surface with a crossing band is not in Menasco form.

\begin{figure}
\begin{center}
	\scalebox{.7}{\includegraphics{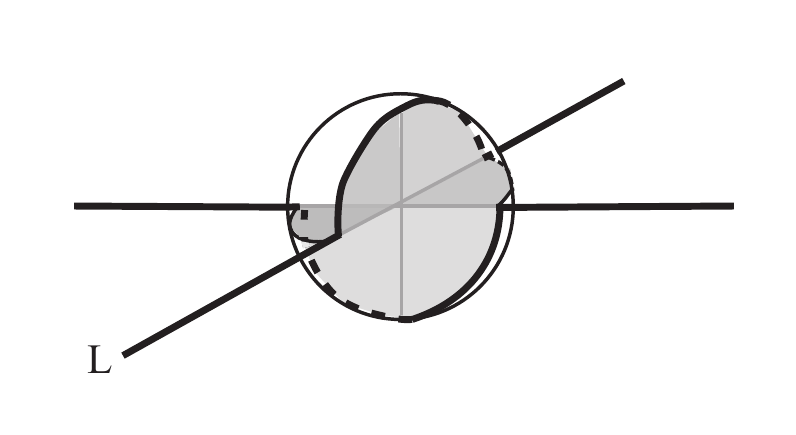}}
	\caption{A crossing band in a Menasco ball.}\label{Fi:desirednew}
\end{center}
\end{figure}

\begin{figure}
\begin{center}
	\scalebox{.7}{\includegraphics{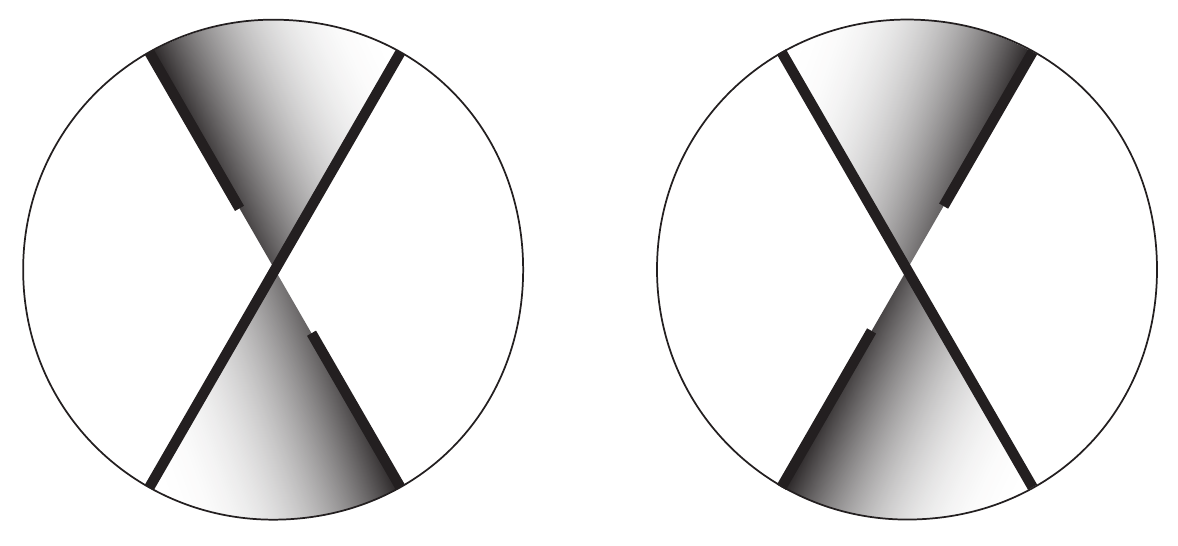}}
	\caption{Crossing bands.}\label{Fi:desired}
\end{center}
\end{figure}  

\begin{lemma}   \label{lemma:desired}    

A meridianally essential surface $S$ spanning an alternating link $L$ can be  
isotoped relative to a given nontrivial Menasco projection $P$ to obtain a crossing band.
\end{lemma}

\begin{proof}For contradiction, choose $S$ and $P$ so that no Menasco ball contains a crossing band and so that $S$ cannot be isotoped to make it so.  Embed $S$ cleanly. No flat or de-alternating intersection arc  
will be parallel to a crossing in $P$.  Otherwise, we could have  
isotoped the neighborhood of that arc in $S$ into the parallel  
crossing, creating a crossing band, as in Figure \ref{Fi:getdesired}. 
\begin{figure}
\begin{center}
	\scalebox{.7}{\includegraphics{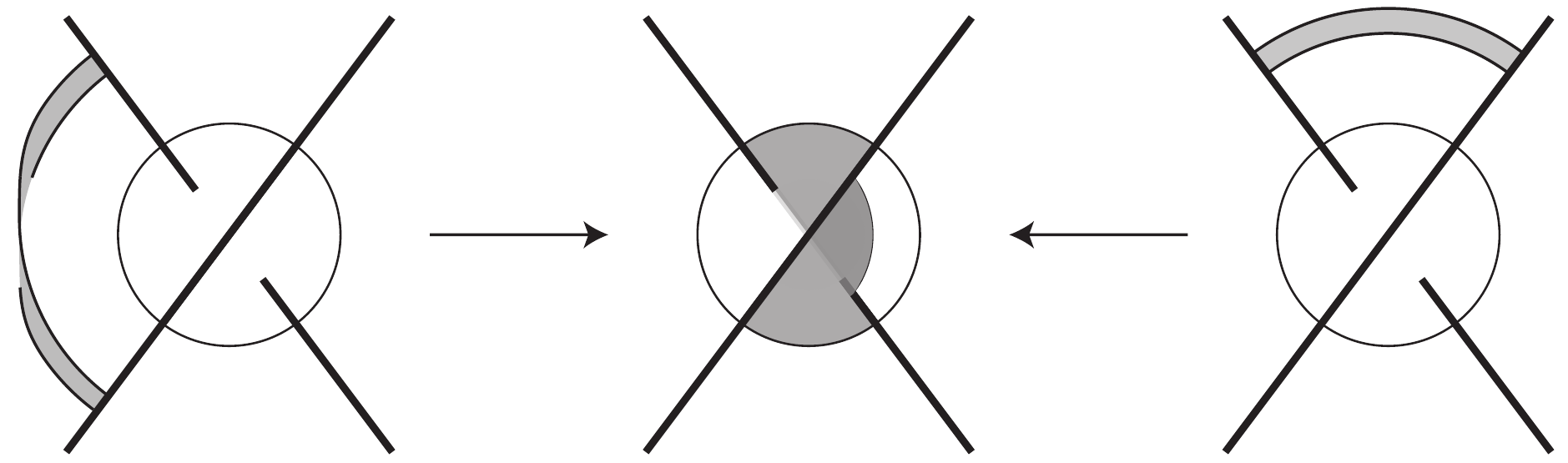}}
	\caption{Isotoping neighborhoods of flat and de-alternating arcs to obtain a crossing band.}\label{Fi:getdesired}
\end{center}
\end{figure}

Therefore, the only type of  
intersection arc that can be parallel to a crossing in $P$ is a re-alternating arc.  Let $A=\cup_{i=1}^{r}{A_i}$ be the collection of all re-alternating  
arcs in $S \cap F$ that are parallel to crossings.  Note that two such arcs may be  
parallel to the same crossing, perhaps even on the same side of the  
crossing, thus making them parallel to each other.

Now consider $R$, one of the $n+r+2$ regions in $F \backslash P^*$, where $P^* = P \cup A$.  Let $B = R \cap S$.

We claim that $\partial R$ contains at least two boundary segments  
which do not contain an endpoint of an arc in $B$.  If $B$ is empty, this  
follows trivially, since the boundary of $R$ must contain at least two  
boundary segments, by Proposition \ref{Prop:boundary}.  If $B$ is not empty, $R$ must contain at least two  
outermost subregions, each disjoint on its interior from $S$ and having a boundary consisting of exactly one arc in $B$ together with portions of at least two boundary segments of $P^*$.  

Suppose the boundary of some outermost subregion $Q$ contains portions of exactly two boundary segments of $P^*$.  Each of the endpoints of the arc in $B$ will lie on the interior of one of these boundary segments, so each boundary segment has exactly one endpoint on $\partial Q$.  At these endpoints they will both connect either to the same arc in $A$ or to the same Menasco ball boundary, either of which will necessarily be parallel to the arc in $B$.  Since all arcs in $A$ are parallel to crossings, the arc in $B$ must therefore be parallel to a crossing, giving a contradiction.  Therefore, the boundary of every outermost subregion contains portions of at least three boundary segments of $P^*$.  Two of these boundary segments will contain an endpoint of an arc in $B$ and will lie on the boundary of adjacent subregions as well.  All other boundary segments will lie entirely on the boundary of this subregion and will not contain any intersection arc endpoints.  Since either $B$ is empty or $R$ contains at least two outermost subregions, it follows that the boundary of every region of $F \backslash P^*$ contains at least two boundary arcs which contain no intersection arc endpoints.

Since $F \backslash P^*$ consists of $n+r+2$ regions, we have at least $2n+2r+4$  
boundary segments which \emph{do not} contain intersection arc endpoints.  It  
follows that we have at most $(4n+4r)-(2n+2r+4) = 2n+2r-4$ boundary  
segments (and certainly no more segments than this) that \emph{do} contain  
intersection arc endpoints. Since each of the $2n+2r$ segments of $P^* 
$ must contain such an endpoint, we have a contradiction.
\end{proof}

\begin{cor}       \label{cor:connected}
Any spanning surface for a non-splittable, alternating link is connected.
\end{cor}

\begin{proof} This is immediate in the case of the trivial knot. Assume it to be true for all non-splittable alternating links of no more than $n$ crossings. Let $L$ be a non-splittable alternating link of $n+1$ crossings and let $S$ be a spanning surface. Then either $S$ is meridianally essential, or a series of compressions/meridianal boundary compressions takes $S$ to a meridianal essential spanning surface, $S'$. Choose a reduced alternating projection of $L$, and put $S'$ in Menasco form relative to that projection. Lemma \ref{lemma:desired} shows that $S'$ can be isotoped to have a crossing that contains a crossing band. Cutting open the link and surface at that crossing results in a spanning surface $S''$ for a non-splittable alternating link $L'$ with fewer crossings. Note that non-splittability of $L'$ utilizes a theorem from \cite {M1}. Since $S''$ is connected, so is $S'$ and hence $S$. \end{proof}

\begin{lemma}        

Any surface $S$ spanning the unknot with a boundary slope  
of $2k$ has $\chi = 1-|k|-2p$ for some $p \in \mathbb{N} \cup \{0\}$.

\end{lemma}

\begin{proof} Note that we have already said that the boundary slope of a spanning surface $S$ must be even. A spanning surface for the trivial knot can be thought of as a surface properly embedded in a solid torus $V$, which is the complement of the trivial knot. That surface will have a boundary on the torus boundary that wraps once meridianally and $2|k|$ times longitudinally around the torus. Compress the surface until it becomes an incompressible surface $S'$. If $k=0$, the resulting surface must be a meridian disk. Otherwise, $S'$ must be nonorientable, since the only incompressible orientable surface with one boundary component properly embedded in the solid torus is a disk. 

By Proposition 1, Corollary 3 and Remark 1 of \cite{Tsau}, $S'$ must have Euler characteristic $1 - |k|$. Thus $S$ has Euler characteristic
$\chi = 1-|k|-2p$ for some $p \in \mathbb{N} \cup \{0\}$.
\end{proof}

Note also that a spanning surface for a knot is necessarily connected. We now proceed to the proof of the main theorem, which we restate here.

\begin{theorem}\label{main} 

Given an alternating projection $P(L)$ and a surface $S$ spanning $L$, we can construct a surface $T$ spanning $L$ with the same genus, orientability, and aggregate slope as $S$ such that $T$ is a basic layered surface with respect to $P$, except perhaps at a collection of added crosscaps and/or handles. When $S$ is orientable, $T$ can be chosen to be orientable with respect to the orientation that $L$ inherits from $S$. 
\end{theorem}

\begin{proof}

We show this to be true for the trivial projection and proceed by induction on crossings in $P$.  

Let $S$ be an orientable surface spanning the unknot.  Since $S$ is orientable, either $\chi(S)=1$ or there exists a series of compressions $C_1, C_2, \ldots , C_q$ that takes $S$ to a disk.  Since compression does not affect boundary slope, $S$ must have a slope of zero.  The previous lemma then implies that $\chi(S)=1-2p$.  Adding $p$ handles to a disk  then creates a spanning surface with the same Euler Characteristic, slope, and orientability as $S$.

Now suppose that $S$ is a nonorientable surface spanning the unknot with a slope of $2k$ and an Euler Characteristic of $1-|k|-2p$.  Then, adding to the disk $k+p$ of the appropriate type of crosscap and $p$ of the other type of crosscap constructs a surface with the same Euler Characteristic, slope, and orientability as $S$.  Thus, Theorem \ref{main} holds for the trivial projection of the unknot.

\bigskip

We now proceed with the inductive proof of Theorem \ref{main}  
assuming it to be true for all Menasco projections with fewer than $n$ crossings. If $P$ is not reduced, we can reduce it and apply our inductive hypothesis to get a surface with the same characteristics as $S$ that appears as a basic layered surface, perhaps with a collection of added handles and/or crosscaps.  We can then carry this surface along as we un-reduce to $P$.  Assume therefore that $P$ is reduced.  

Suppose that our reduced projection $P$ has $n$ crossings and that $w(P)=y$ and $lk(L)=z$.  Also suppose that a surface $S$ spans $L$ with an aggregate slope of $l$ and an Euler Characteristic of $x$.  Recall that $\tau(S,P)=l-(y-2z)$.  Begin by performing a series of compressions/meridianal boundary compressions taking $S$ to a meridianally essential surface $S'$, also spanning $L$.  By Lemma \ref{lemma:desired}, we can isotope $S' $ to have a crossing band at a crossing in $P$.  Without loss of  
generality, suppose that this is a positively twisted 
crossing.  Cutting $S'$ at this crossing as in Figure \ref{Fi:cut}
 produces a new  
surface $S^*$ spanning a link $L^*$ in an alternating projection $P^*$ with $n-1$ crossings such that $\tau(S^*,P^*)=l-(y-2z)-1$ and $\chi(S^*)=x+1$.   Notice that this projection may not be reduced. Reduce it to a projection $P^{**}$.

By our inductive hypothesis, we can create a surface with the same  
characteristics as $S^*$ that appears layered with respect to $P^{**}$,  
except perhaps at a collection of added handles or crosscaps.  We construct such a surface, stacking all inner disks above any outer disks containing them. 

Call this surface $T^*$.  We can then carry $T^*$ along as a layered surface stacked in this manner while we un-reduce from $P^{**}$ to $P^*$. Recall that $\chi(T^*)=\chi(S^*)=x+1$.  Also, since $T^*$ and $S^*$ have the same aggregate slope, $\tau(T^*,P^*)=\tau(S^*,P^*)=l-(y-2z)-1$.  At the removed crossing, we then glue to $T^*$ a crossing band, which we know to be possible because of the particular layering of $T^*$.  This gives a layered surface $T'$ spanning $L$ with $\chi(T')=(x+1)-1=x$ and $\tau(T,P)=[l-(y-2z)-1]+1=l-(y-2z)$.  Since $\tau(T',P)=\tau(S',P)$, we know that $l(T',L)=l(S',L)$. Adding the appropriate handles/crosscaps, corresponding to the initial compressions/meridianal boundary compressions of $S$ produces a surface $T$ spanning $L$ that appears layered with respect to $P$, except at a collection of handles or crosscaps, and that has the same genus and aggregate slope as $S$. 

\begin{figure}
\begin{center}
	\scalebox{.3}{\includegraphics{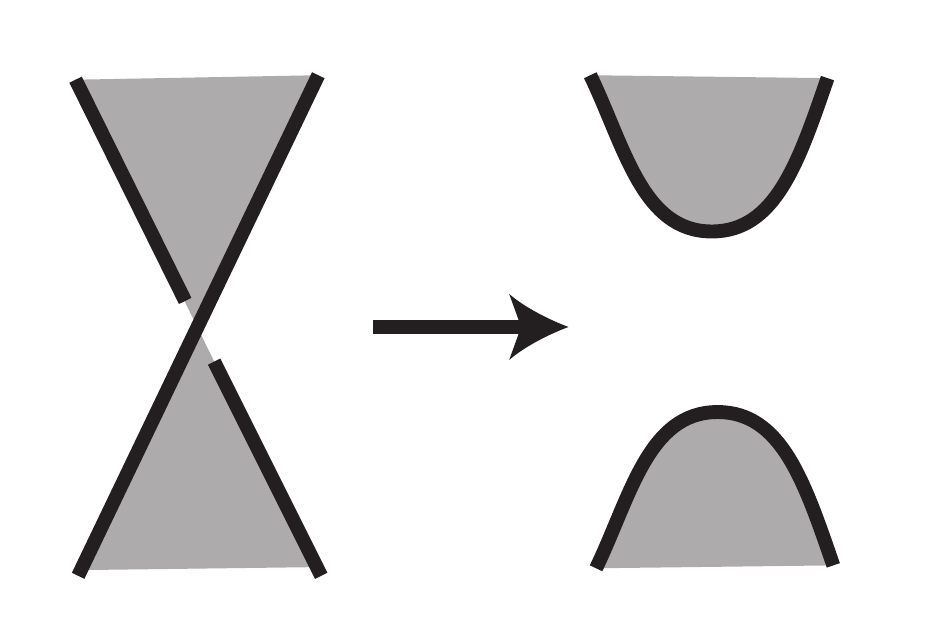}}
	\caption{Cutting a crossing containing a crossing band yields a new surface. The process can also be reversed.}\label{Fi:cut}
\end{center}
\end{figure}

It remains to show that $T$ has the same orientability as $S$. If $S^*$ and hence $T^*$ are nonorientable, then so are $S$ and $T$. Also, if there are any meridianal boundary compressions in going from $S$ to $S'$, then both $S$ and $T$ are nonorientable. So assume instead that $S^*$ and $T^*$ are both orientable, and no meridianal boundary compressions occurred in going from $S$ to $S'$. We have two cases to consider, as appearing in Figure \ref{Fi:GluingOrientations}, depending on the orientation of $L^*$ near the removed crossing band. In the case depicted in Figure \ref{Fi:GluingOrientations} (a), we can extend the orientation on $L^*$ to an orientation on $L$, and we can assign a normal direction to the crossing band as we glue it back in that agrees with the normal direction on both $S^*$ and $T^*$ to both sides of the crossing. Thus, both $S'$ and $T'$ will be orientable  with respect to this orientation on $L$, as will $S$ and $T$. Similarly, as in the case depicted in Figure \ref{Fi:GluingOrientations} (b), if we were to glue in an {\it untwisted} band to both $S^*$ and to $T^*$ as appears in Figure \ref{Fi:GluingOrientations} (c), we would obtain a pair of orientable surfaces. Corollary \ref{cor:connected} implies that each such surface contains an annulus that passes through this untwisted band once. Replacing the untwisted band with the actual crossing band to obtain $S'$ and $T'$ will result in the fact each contains a Mobius band, implying that both $S'$ and $T'$ are nonorientable, and therefore implying the same for $S$ and $T$.

\begin{figure}
\begin{center}
	\scalebox{.7}{\includegraphics{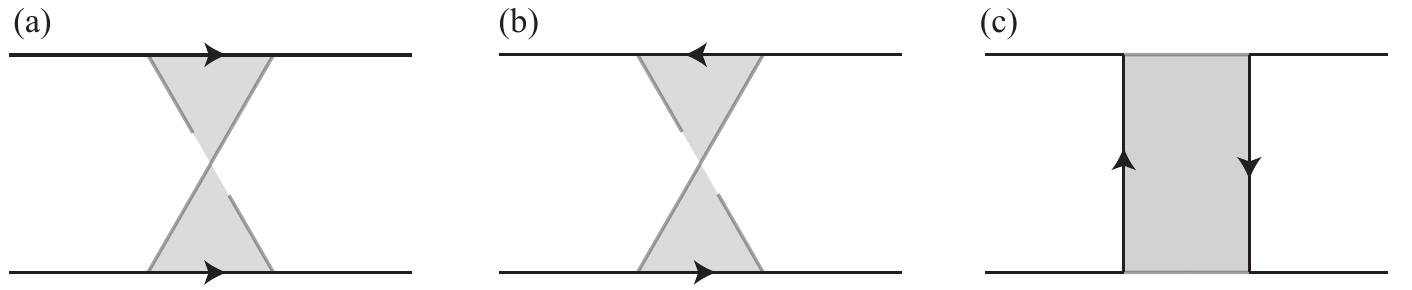}}
	\caption{(a) and (b): Up to symmetry, the two possibilities for the orientation on $L$ near the removed crossing band. (c): In the case of (b), gluing an untwisted band extends the orientation on $L^*$ to one on $L$.}\label{Fi:GluingOrientations}
\end{center}
\end{figure}

Thus, we have proven that our inductive construction preserves orientability as well as Euler Characteristic and aggregate slope. 
\bigskip

All we have left to show is that our layered surface $T$ can be a \emph{basic} layered surface (potentially with added crosscaps or handles), one in which no crossing connects a circle to itself.  Suppose $T$ is a non-basic layered surface such that in its construction, we stacked all inner disks above outer ones, as described earlier.  $T$ is then meridianally boundary-compressible.  Performing this meridianal boundary compression will take $T$ to a layered surface in which fewer crossings connect any circle to itself than in $T$.  Performing in turn all possible meridianal boundary compressions of this sort and pairing each with the addition of an appropriate crosscap will create a new layered surface (with added crosscaps and perhaps handles as well) that is in fact basic and that has the same Euler Characteristic, aggregate slope, and orientability as $T$ and therefore as $S$. 
\end{proof}

\section{Implications} \label{S:Implications}

We now provide some corollaries to Theorem \ref{main}.  

\begin{cor}       
Given an alternating projection of a link,  the minimal genus basic layered surface (orientable or nonorientable) generated from that projection is a minimal genus spanning surface for that link. If a minimal genus basic layered surface is nonorientable, it realizes the nonorientable genus for the link. If all minimal genus basic layered surfaces are orientable, the nonorientable genus is 1/2 greater than this minimal genus.\end{cor}

\begin{cor}       

There exists a surface $S$ spanning $L$ in an alternating projection $P$ with Euler Characteristic $x$ and aggregate slope $l$ if and only if there exists a basic layered surface $T$ generated from $P$ with Euler Characteristic $x^*$ and aggregate slope $l^*$ such that $x=x^*+ \frac{|l-l^*|}{2}+2p$ for some $p \in \mathbb{N} \cup \{0\}$.  Furthermore, $S$ can be nonorientable     if and only if \emph{either} such a $T$ is nonorientable, $l \neq l'$, or $p>0$.  $S$ can be orientable if and only if such a $T$ is orientable and $l=l^*$.
\end{cor}

Note that in \cite {MT}, the authors point out that no example is known of two mutant knots with distinct crosscap number, unlike what occurs for orientable genus, where the Kinoshita-Terasaka mutants provide just such an example. It is known that for alternating knots and links, orientable genus is preserved by mutation.

\begin{cor}
If a knot is alternating, any mutant of it has the same crosscap number as it does.
\end{cor}

\begin{proof} This follows from Theorem \ref{main} together with the fact that a mutant of an alternating knot is also alternating, and must come from a sequence of mutations. Each of the mutations is  along a 4-punctured sphere that  intersects the projection plane in either a single circle intersecting the knot at four points or in two circles, each intersecting the knot at two points such that the sphere has two saddles. This follows from \cite{M1}. The collection of genera paired with orientability of spanning surfaces of the knot is preserved by mutancy along such 4-punctured spheres.
\end{proof}

 In \cite{HT}, the authors prove that a 2-bridge knot cannot have two minimal nonorientable genus  spanning surfaces, one boundary-incompressible and one boundary-compressible. They ask whether or not this is true in general. We answer that question in the negative with the following example.

{\bf Example.} Consider the knot and surfaces that appear in Figure \ref{Fi:knot}. By applying the algorithm of Section \ref{Sec:layered}, we find that a spanning surface $S$ with greatest possible Euler characteristic is orientable of genus 2, and therefore the  nonorientable genus is 5/2, a surface for which can be obtained by adding a crosscap to the surface $S$. This surface $S_1$, which appears at the top of Figure \ref{Fi:knot}, is incompressible, as any compression would yield a spanning surface with genus less than the minimum 2, and it is obviously boundary-compressible. However, we can also obtain the second surface $S_2$ appearing in Figure \ref{Fi:knot} by sliding one strand around the far side of the projection sphere and changing one crossing split. It is a non-basic layered surface. By situating the bands connecting top and bottom to be alternately over and under the main disk, we obtain a surface that is boundary-incompressible.

\begin{figure}
\begin{center}
	\scalebox{.5}{\includegraphics{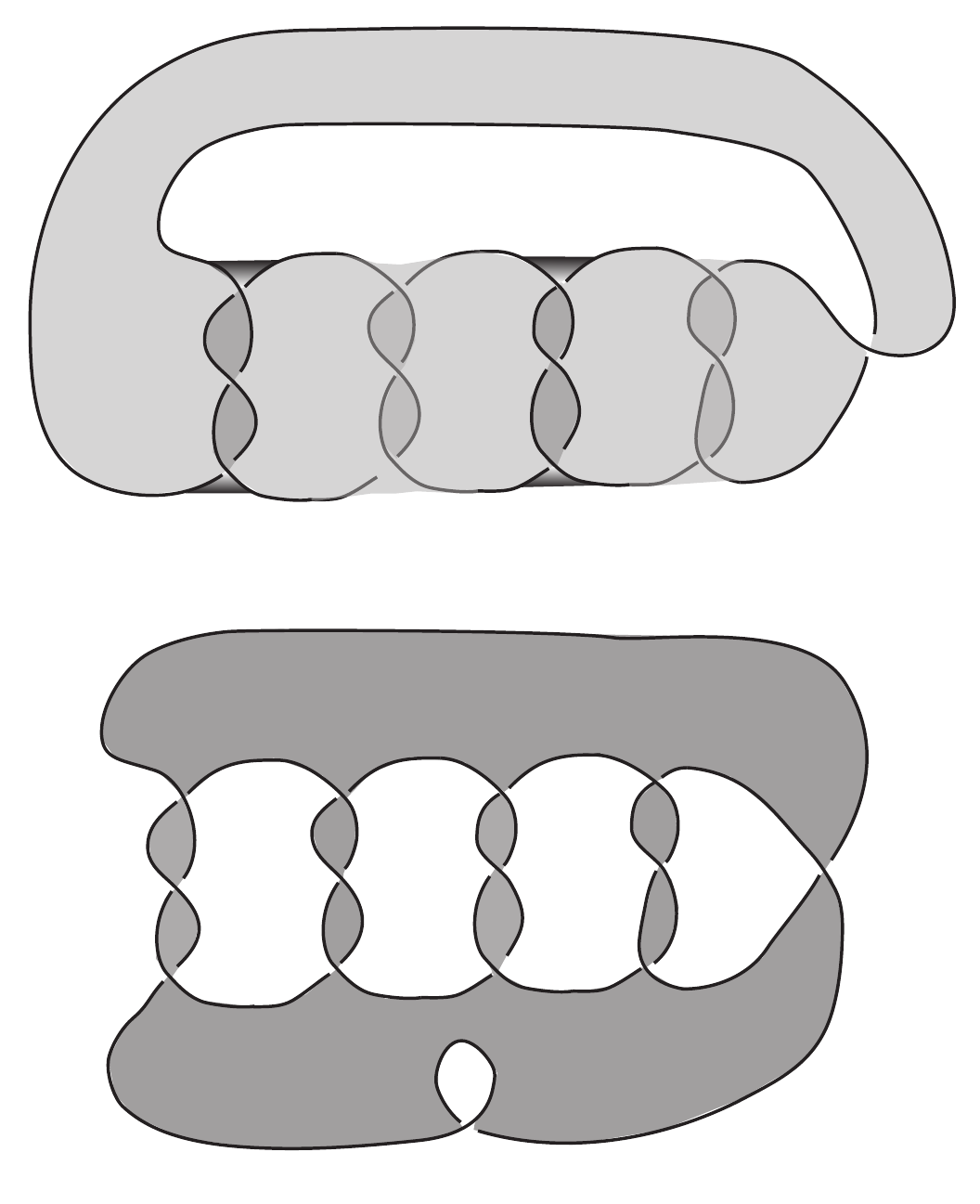}}
	\caption{An example of a knot such that it has minimal genus nonorientable surfaces, one that is  boundary-compressible and one that is  boundary-incompressible.}\label{Fi:knot}
\end{center}
\end{figure}  

\begin{lemma} The surface $S_2$ is incompressible and boundary-incompressible.
\end{lemma}

\begin{proof} That the surface is incompressible follows immediately from the fact that a compression would yield a spanning surface of genus less than 2, contradicting the fact that 2 is the minimal genus possible for this knot. To see that it is boundary-incompressible, we intersect it with a sphere $Q$ intersecting the projection plane in the circle shown in Figure \ref{Fi:boundary}.

\begin{figure}
\begin{center}
	\scalebox{.5}{\includegraphics{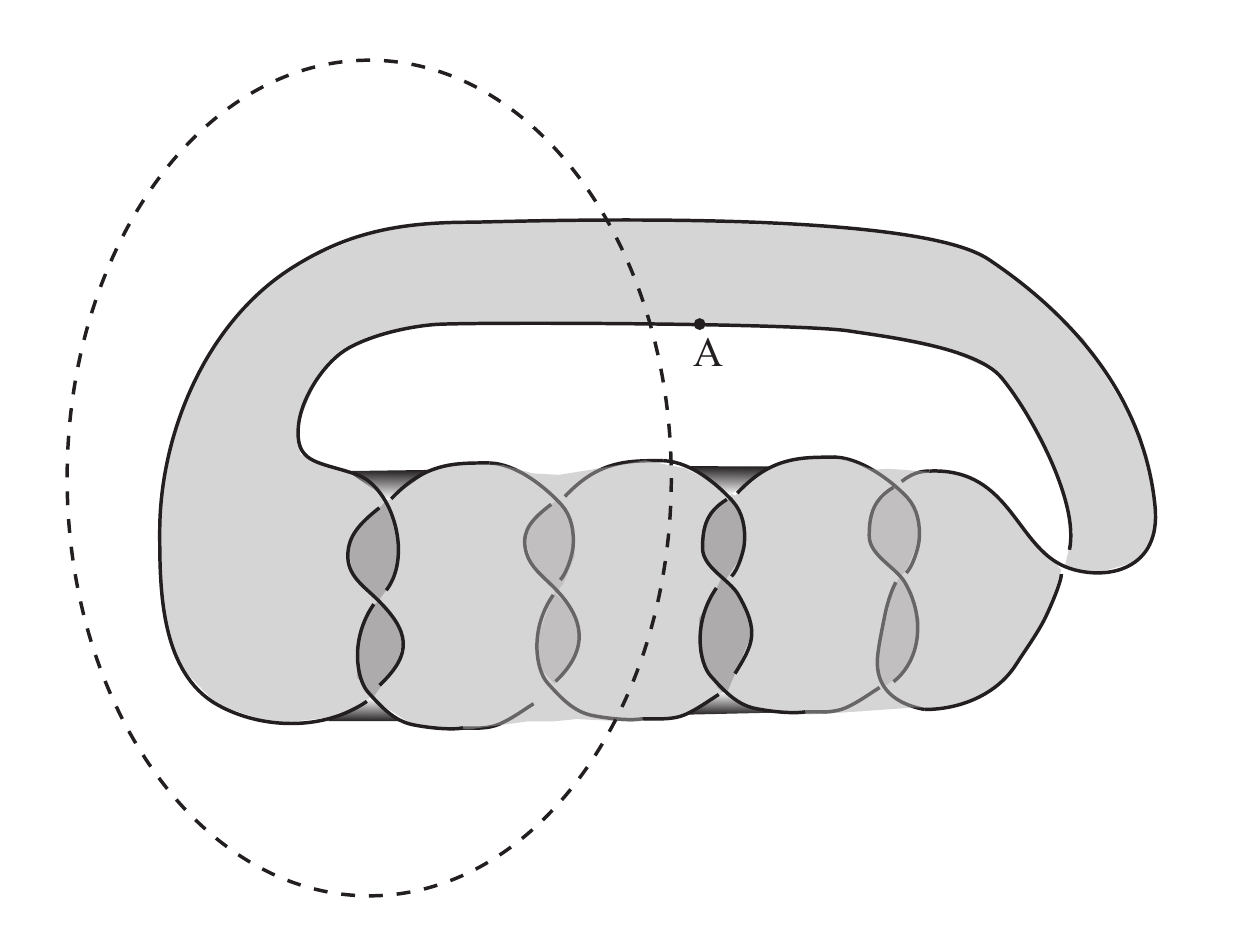}}
	\caption{Proving boundary-incompressibility of $S_2$.}\label{Fi:boundary}
\end{center}
\end{figure}  

Suppose $S_2$ is boundary compressible.  We can choose a boundary-compressing disk $D$ so that  the two points of intersection of its boundary with $\partial S_2$ occur at or very near the point $A$ in the figure and so that it intersects $Q$ in a minimal number of components. Then, an outermost intersection arc of $D \cap Q$ on $D$ cuts a disk $D'$ from $D$, the boundary of which is made up of an arc in $Q$ and an arc in $\partial N(K) \cup S_2$. This implies that to one of the sides of $Q$, the interior of this disk $D'$ lies in the interior of a handlebody $H$ obtained by removing $S_2 \cup N(K)$ from one of the  balls bounded by $Q$.  The boundary of $D'$ consists of a single arc in $Q$ and a second arc, for which there are three cases. The second arc is either entirely contained in $S_2$, entirely contained in $\partial N(K)$, or it is made up of two arcs, one in $S_2$ and one in $\partial N(K)$.  Because the fundamental group of a handlebody is a free product, in the first two cases, the arc in $\partial D' \backslash (\partial D' \cap Q)$ cannot pass over either of the two bands formed from the triple half-twists without immediately turning around and retracing its path, and it must therefore be isotopic on $S_2 \cup \partial N(K)$ to a path that passes directly from left to right, beginning and ending on $Q$. However, the two twisted bands prevent such an arc from being isotopic into $Q$ while fixing its endpoints.

In the third case, the fact we can only jump once from $S_2$ to $\partial N(K)$ as we travel around $\partial D'$ and those jumps occur in the vicinity of the point A, again implies that the triple twisting of the two bands prevents the arc $\partial D' \backslash (\partial D' \cap Q)$ from traversing one of the triple-twisted bands. Again, the arc must  must be isotopic on $S_2 \cup \partial N(K)$ to a path that passes directly from left to right, beginning and ending on $Q$,the twisted bands of which again prevent it from being isotopic into $Q$ while fixing its endpoints, thereby proving $S_2$ is boundary-incompressible.

\end{proof}

\pagebreak

\section{Appendix}

We conclude with a list of nonorientable genera of prime alternating knots through 9 crossings.

\begin{table}[h]
\begin{center} 
  \begin{tabular}{|| c | c || c | c || c |  c || c | cc ||}
\hline
Knot	    & Nonor. genus &   Knot  & Nonor. genus & Knot & Nonor. genus  & Link& Nonor. genus \\ \hline
 $3_1$  & 1/2      &                    $8_{12}$ &    2    &              $9_{19}$ &    2  & $2_1^2$ & 1/2 \\ \hline
  $4_1$ &     1     &                       $8_{13}$ &    3/2    &         $9_{20}$ &   3/ 2  & $4_1^2$ & 1/2   \\ \hline
   $5_1$  &  1/2   &              $8_{14}$ &    2    &         $9_{21}$ &    2   & $5_1^2$ & 1/2 \\ \hline
   $5_2$ &    1   &                  $8_{15}$ &    2    &         $9_{22}$ &    3/2 & $6_1^2$ & 1/2 \\ \hline
   $6_1$ &    1  &                  $8_{16}$ &    3/2    &         $9_{23}$ &    2  & $6_2^2$ &  1/2 \\ \hline 
   $6_2$ &    1   &                  $8_{17}$ &    2    &         $9_{24}$ &    2  & $6_3^2$ & 1 \\ \hline
      $6_3$ &   3/2  &            $8_{18}$ &    2    &      $9_{25}$ &    2  & $7_1^2$ &1/2 \\ \hline
   $7_1$ &   1/2  &                $9_1$ &    1/2    &      $9_{26}$ &    2  & $7_2^2$ &1  \\ \hline
     $7_2$ &    1  &                 $9_2$ &    1    &      $9_{27}$ &    2    & $7_3^2$ &1 \\ \hline
  $7_3$ &    1   &                  $9_3$ &    1    &        $9_{28}$ &    2    & $7_4^2$ &1/2 \\ \hline
  $7_4$ &    3/2   &                $9_4$ &    1    &      $9_{29}$ &    3/2   & $7_5^2$ &1 \\ \hline   
  $7_5$ &     3/2    &               $9_5$ &    3/2    &      $9_{30}$ &    2    & $7_6^2$ &1  \\ \hline
  $7_6$ &      3/2    &                 $9_6$ &    3/2    &       $9_{31}$ &    2    & $8_1^2$ &1/2 \\ \hline   
  $7_7$ &     3/2   &                  $9_7$ &    3/2    &       $9_{32}$ &    2   & $8_2^2$ &1/2 \\ \hline   
  $8_1$ &    1   &                        $9_8$ &    3/2    &        $9_{33}$ &    2    & $8_3^2$ &1 \\ \hline
  $8_2$ &     1     &                       $9_9$ &    3/2    &        $9_{34}$ &    2    & $8_4^2$ &1 \\ \hline  
  $8_3$ &     3/2    &                     $9_{10}$ &    3/2    &        $9_{35}$ &    3/2   & $8_5^2$ &1 \\ \hline  
  $8_4$ &     1    &                        $9_{11}$ &    3/2    &        $9_{36}$ &    3/2   & $8_6^2$ &1  \\ \hline
  $8_5$ &     1     &                $9_{12}$ &    3/2    &        $9_{37}$ &    2    & $8_7^2$ &1 \\ \hline
  $8_6$ &     3/2    &                 $9_{13}$ &    3/2    &        $9_{38}$ &    2   & $8_8^2$ &3/2  \\ \hline  
  $8_7$ &     3/2    &                  $9_{14}$ &    3/2    &        $9_{39}$ &    2   & $8_9^2$ &1  \\ \hline 
  $8_8$ &    3/2    &                 $9_{15}$ &    2    &        $9_{40}$ &    2   & $8_{10}^2$ &1 \\ \hline  
  $8_9$ &    3/2    &                 $9_{16}$ &    3/2    &        $9_{41}$ &    3/2   & $8^2_{11}$ &1  \\ \hline
  $8_{10}$ &     3/2    &               $9_{17}$ &    3/2    &         &       & $8^2_{12}$ &3/2 \\ \hline  
  $8_{11}$ &    3/2    &             $9_{18}$ &    2    &       &                       & $8^2_{13}$ &3/2 \\  \hline
     &   &    &  &  &  & $8^2_{14}$ &  3/2 \\  \hline

  \end{tabular}
   \caption{Nonorientable genus for prime alternating knots through nine crossings and two-component links through eight crossings.}
\end{center}
\end{table}

\end{document}